\documentclass[microtype]{gtpart}
\usepackage{amsfonts, amsmath, amssymb, amsthm, color,
  bookmark,graphicx,subfig}

\allowdisplaybreaks

\newtheorem{thm}{Theorem}[section]
\newtheorem{prop}[thm]{Proposition}
\newtheorem{lemma}[thm]{Lemma}
\newtheorem{cor}[thm]{Corollary}

\theoremstyle{definition}
\newtheorem{defn}[thm]{Definition}
\newtheorem{question}[thm]{Question}
\theoremstyle{remark}
\newtheorem{ex}[thm]{Example}
\newtheorem{rem}[thm]{Remark}

\newcommand{\ga}{\Gamma}

\newcommand{\e}{\varepsilon}
\renewcommand{\d}{\mathrm{d}}
\newcommand{\tr}{{\rm tr}}

\DeclareMathOperator{\comm}{Comm}

\newcommand{\cat}{{\rm CAT(0)}}

\newcommand{\cN}{\mathcal{N}}
\newcommand{\cA}{\mathcal{A}}
\newcommand{\cB}{\mathcal{B}}
\newcommand{\cC}{\mathcal{C}}

\newcommand{\cM}{\mathcal{M}}

\newcommand{\cL}{\mathcal{L}}

\newcommand{\tildephi}{{\tilde\phi}}

\newcommand{\A}{\mathbb{A}}
\newcommand{\E}{\mathbb{E}}
\newcommand{\N }{\mathbb{N}}
\newcommand{\T}{\mathbb{T}}

\newcommand{\isom}{\mathrm{Isom}}

\def\coloneq{\mathrel{\mathop\mathchar"303A}\mkern-1.2mu=}

\begin{document}

%\keywords{Commensurating HNN-extensions, biautomatic groups, \cat{} groups}
%\subjclass[2010]{20F67, 20F10, 20E06}
\subject{primary}{MSC2000}{20F67}
\subject{secondary}{MSC2000}{20F10}
\subject{secondary}{MSC2000}{20E06}

\title{Commensurating HNN-extensions: non-positive curvature
and biautomaticity}

\author{Ian J. Leary}
%\thanks{}
\author{Ashot Minasyan}
\address{CGTA, School of Mathematical Sciences,
University of Southampton, Highfield, Southampton, SO17~1BJ, United
Kingdom.}
\email{I.J.Leary@soton.ac.uk, aminasyan@gmail.com}

\begin{abstract}
We show that the commensurator of any quasiconvex abelian subgroup
in a biautomatic group is small, in the sense that it has
finite image in the abstract commensurator of the subgroup.  Using this
criterion we exhibit groups that are \cat{} but not biautomatic.
These groups also resolve a number of other questions concerning
\cat{} groups.
\end{abstract}

\maketitle

\section{Introduction}
The theory of automatic and biautomatic groups was developed in the late
1980's, and described in the book \cite{Eps} written by Epstein, Cannon, Holt, Levy, Paterson and Thurston.  The question of whether
there are any automatic groups that are not biautomatic appears in
this book as~\cite[Question~2.5.6]{Eps} and as Remark~6.19 at the end
of the paper of Gersten and Short \cite{G-S} in which biautomatic
groups were introduced.
By that time it had been shown that hyperbolic groups
are biautomatic (implicit in \cite[Theorem~3.4.5]{Eps}).
The definition of biautomaticity has both language-theoretic
and geometric aspects, whose interaction is non-trivial.
This motivated Alonso and Bridson to introduce, in the early
1990's, the geometric class of semihyperbolic groups~\cite{alobri}.
This class contains all biautomatic groups and all \cat{} groups.
In the mid 1990's Niblo and Reeves  proved that \cat{}
cubical groups are biautomatic~\cite{nibree}.  The question of whether
all \cat{} groups are automatic or even biautomatic must have been a
motivating question for much of the above work, but the earliest written
versions that we were able to find were in the PhD thesis of Elder \cite[Open Question~2]{Elder}, written in 2000, and in a problem list compiled
by McCammond after the American Institute of Mathematics meeting
`Problems in Geometric Group Theory' April 23--27,
2007~\cite[Question~13]{mccammond} (see also \cite[Section~6.6]{FHT}).
We answer the latter question by constructing the first examples
of \cat{} groups that are not biautomatic.

The groups that we construct are higher-dimensional analogues of
Baumslag-Solitar groups \cite{Baum-Sol}, in the sense that they are HNN-extensions
of free abelian groups of rank greater than one in which the stable letter
conjugates two finite-index subgroups.  Consider the group $G_P$
given by the presentation
\[G_P \coloneq \langle a,b,t \,\|\, [a,b]=1,\,\, ta^2b^{-1}t^{-1}=a^2b,\,\,
tab^2t^{-1}= a^{-1}b^2 \rangle.\]
This is an HNN-extension of  $L=\langle a,b\rangle \cong \Z^2$
in which the stable letter conjugates two subgroups of index five.
If we let $L$ act on the Euclidean plane $\E^2$ in such a way that
$a$~and~$b$ act as translations of length one in orthogonal directions,
the elements $a^2b^{-1}$ and $ab^2$ act as translations of length
$\sqrt{5}$ in orthogonal directions, as do the elements
$a^2b$~and~$a^{-1}b^2$.  The action of $\langle a,b\rangle$ on $\E^2$
extends to an action of the whole group $G_P$, in which the stable
letter acts as a rotation through $\arccos(3/5)$.  The fact that this
action is by isometries is what ensures that $G_P$ is \cat{}, while
the fact that the rotation through $\arccos(3/5)$ has infinite order
is what allows us to show that $G_P$ is not virtually biautomatic (i.e., no finite-index subgroup is biautomatic).

The action of $G_P$ on the Euclidean plane can be used to show that
for any $n\neq 0$, the subgroup $\langle a^n,t^n\rangle<G_P$ is not
abelian.  But also $\langle a^n,t^n\rangle$ contains a finite-index
subgroup of $\langle a,b\rangle$, and so it cannot be a free group (see Corollary~\ref{cor:wise} below).
Thus the elements $a,t\in G_P$ give the first negative answer to a
question of Wise concerning a strong version of the Tits Alternative
for \cat{} groups~\cite[Question~2.7]{Bestvina-list}.

The Bass-Serre tree $T$ for $G_P$ is a regular tree of valency $10$, and
we show that $G_P$ acts geometrically on the direct product $\E^2\times T$.
For $F$ a free group of rank 5, the group $\Z^2\times F$ acts geometrically
on the same \cat{} space as $G_P$, and it follows that these groups are
quasi-isometric to each other. Hence any property that is
not shared by $\Z^2\times F$ and the group $G_P$ cannot be
invariant under quasi-isometry, even amongst \cat{} groups.  In
particular there are \cat{} groups that are quasi-isometric to
$\Z^2\times F$ but are not virtually biautomatic, and hence cannot
be virtually cubical. In fact, by a recent result of Huang and Prytu\l{a}
\cite{huangprytula}, no finite-index subgroup of $G_P$ admits a proper action on a finite-dimensional \cat{} cube complex by cubical automorphisms (because no positive power of $t$ normalizes a subgroup of finite index in the abelian base group $L$).

Note that although $G_P$ is quasi-isometric to $\Z^2\times F$ it is
not commensurable to it, and so $\Z^2\times F$ is not quasi-isometrically
rigid, contrary to some claims in the existing literature.
Moreover, $G_P$ embeds as an
irreducible lattice in the group of isometries of $\E^2\times T$.
There has been some confusion concerning this property in the
literature: in particular \cite{capmon} claims that no such lattices
exist (this has recently been rectified in \cite{capmon-corr}).

By varying the geometry of the free abelian subgroup, one can
construct similar examples in which the indices of the subgroups
conjugated by the stable letter are smaller.  Consider
the groups $G_{k,2}$ for $k\in\Z$ with presentations
\[G_{k,2}\coloneq\langle a,b,t \,\|\, [a,b]=1,\,\, tat^{-1}=b,\,\,
tb^2 t^{-1}= a^{-2}b^k \rangle.\]
In $G_{k,2}$, the stable letter conjugates two index $2$ subgroups
of $\langle a,b\rangle \cong \Z^2$, and since $tat^{-1}=b$, this relator and
the generator $b$ can be eliminated, giving a presentation of
$G_{k,2}$ with just two generators and two relators.
We show that $G_{k,2}$ is \cat{} if and only if $-3\leq k\leq 3$
and that $G_{k,2}$ is biautomatic if and only if
$k\in\{-2,0,2\}$.  Thus the groups $G_{k,2}$ for $k\in\{-3,-1,1,3\}$
are \cat{} but not biautomatic, and the elements $a,t\in G_{k,2}$
give counterexamples to Wise's question for these values of $k$
too~\cite[Question~2.7]{Bestvina-list}.

Although our main examples arise already for base groups free abelian
of rank $2$, we consider commensurating HNN-extensions of free
abelian groups of arbitrary finite rank.  Such a group is
described by a pair $L',L''$ of finite-index subgroups of
$L=\Z^n$, together with a matrix $A\in GL(n,\Q)$ such that
multiplication by $A$ defines an isomorphism
$A\times:L'\to L''$.  We denote this group by $G(A,L')$, because
$L''=A\,L'$ is determined by the pair $(A,L')$.
Many of the results that we obtain
concerning these groups are summarized in the following
theorem.

\begin{thm}\label{thm:ashtalk}
Let $G=G(A,L')$ be a group defined above. Then
\begin{enumerate}
  \item\label{claim:1}  $G$ is residually finite $\Leftrightarrow $
 $G$ is linear $\Leftrightarrow $  either
$L'=L$ or $A\,L'=L$ or $A$ is conjugate in $GL(n,\Q)$ to
    an element of $GL(n,\Z)$;
    \item\label{claim:2} $G$ is \cat{} $\Leftrightarrow $
$A$ is conjugate in $GL(n,\R)$ to an orthogonal matrix;
    \item\label{claim:3}  $G$ is biautomatic $\Leftrightarrow $  $G$ is virtually biautomatic
 $\Leftrightarrow $ $A$ has finite order.
\end{enumerate}

\end{thm}

In the special case when $n=1$ the three parts of Theorem~\ref{thm:ashtalk}
are previously known results concerning Baumslag-Solitar groups.
Similarly to Baumslag-Solitar groups, many of the groups
$G(A,L')$ can be shown to be non-Hopfian.  We give a criterion
for this in Proposition~\ref{prop:nonhopf} which implies that
the groups $G_P$ and $G_{k,2}$, for $k$ odd, are all
non-Hopfian.

The three parts of Theorem~\ref{thm:ashtalk} are proved separately. The characterization of when $G$
is residually finite in claim \eqref{claim:1} follows from earlier work in \cite{A-R-V}, and we
use the affine action of $G$ on $L\otimes \Q$ to show that each residually
finite $G$ is linear over $\Q$. For claim \eqref{claim:2} we introduce an addendum to the
Flat Torus Theorem concerning the commensurator of an abelian group of
semi-simple isometries of a \cat{} space, which may be of independent
interest. The main result used to prove claim  \eqref{claim:3} is the next theorem that imposes strong restrictions on the commensurator of a quasiconvex
abelian subgroup in a biautomatic group.

\begin{thm}\label{thm:comm_of_ab_qc_sbgp} Suppose that $G$ is a group with a biautomatic structure $(\cA,\cL)$,
and $H \leqslant G$ is an $\cL$-quasiconvex abelian subgroup. Then the commensurator $\comm_G(H)$, of $H$ in $G$, has finite image in the abstract commensurator $\comm(H)$. In particular, there is a finite-index subgroup
 $\comm_G^0(H) \lhd \comm_G(H)$ such that every finitely generated subgroup of $\comm_G^0(H)$
centralizes a finite-index subgroup of $H$ in $G$.
\end{thm}

The key tool used in the proof of Theorem~\ref{thm:comm_of_ab_qc_sbgp} is the boundary of an automatic structure, discussed in Section~\ref{sec:background}.
Our strategy is to show that a biautomatic structure on $G$ induces a biautomatic structure
on the quasiconvex subgroup $H$, and $\comm_G(H)$ acts on the boundary of this structure in a natural way (see Section~\ref{sec:act_of_comm}).  In Section~\ref{sec:bd-abelian} we show that
the boundary of any (bi)automatic structure on the abelian group $H$ is finite, hence a finite-index subgroup of $\comm_G(H)$ must fix this boundary pointwise.
Finally, in Section~\ref{sec:abgp} we apply the latter to prove Theorem~\ref{thm:comm_of_ab_qc_sbgp}.

The technical heart of this paper is the results concerning biautomaticity,
but the other parts of the paper may be read independently of this
material.  Section~\ref{sec:flat_torus} contains our addendum to the Flat Torus Theorem.  Section~\ref{sec:comm_HNN} introduces the
groups $G(A,L')$, and characterizes which of them
are \cat{}.  Section~\ref{sec:biaut} characterizes which of the
groups $G(A,L')$ are biautomatic.  Section~\ref{sec:2dim} considers
in more detail the case when $L=\Z^2$, and discusses
a class of examples which includes the groups $G_P$
and $G_{k,2}$ already mentioned above,
establishing many of their properties.
Section~\ref{sec:n-H} concerns the non-Hopfian property,
with results only in the case $L=\Z^2$, and residual
finiteness, with a more general result.
Section~\ref{sec:amalgams} shows that many of our
examples can be embedded as index two subgroups of
free products with amalgamation in which each factor
is virtually abelian. This construction gives rise to amalgamated products of virtually  abelian groups with surprising properties. Section~\ref{sec:open-q} concludes with a short
list of open problems concerning the groups $G(A,L')$.

\subsection*{Acknowledgements}
Firstly, the authors thank Tomasz Prytu\l{a}.  He asked the authors
for an example of an abelian subgroup of a \cat{} group, whose
commensurator does not normalize any finite-index subgroup, in
connection with his work on classifying spaces for families of abelian
subgroups~\cite{osajdaprytula,prytula}.  This question was what
originally led us to consider the groups $G'_P$ and $G_P$.  The
authors also thank Martin Bridson, Pierre-Emmanuel Caprace, Derek
Holt, Jingyin Huang, Denis Osin and Kevin Whyte for helpful comments
on aspects of the work.  The first named author was partially
supported by a Research Fellowship from the Leverhulme Trust.  Most of
the work was done in Southampton.  However, during the project the
first named author spent some weeks at INI, Cambridge, where research
was supported by EPSRC grant EP/K032208/1.

\section{Background and notation}\label{sec:background}
\subsection{Commensurators}\label{subsec:comm}
If $G$ is a group and $H \leqslant G$ is a subgroup, the
  \emph{commensurator of $H$ in $G$} is the subset defined by
\[\comm_G(H)\coloneq \{g \in G \mid |H:(H\cap gHg^{-1})|<\infty
\mbox{ and } |gHg^{-1}:(H\cap gHg^{-1})|<\infty\}.\]
It is not difficult to see that $\comm_G(H)$ is actually a subgroup of $G$.
We will say that $G$ \emph{commensurates} $H$ if $G=\comm_G(H)$.

The elements of the \emph{abstract commensurator} of a group $G$,
denoted $\comm(G)$, are equivalence classes of isomorphisms
$\phi:H\rightarrow K$, where both $H$ and $K$ are finite-index
subgroups of $G$.  Two such isomorphisms $\phi_1:H_1\rightarrow K_1$
and $\phi_2:H_2\rightarrow K_2$ are equivalent if there is $H\leqslant H_1\cap H_2$
also of finite index in $G$ so that $\phi_1|_H= \phi_2|_H$.  The
abstract commensurator is a group, in which the composite of the
equivalence classes of $\phi:H\rightarrow K$ and $\psi:L\rightarrow M$
is the class of
$\psi\circ\phi:H\cap \phi^{-1}(L)\rightarrow \psi(\phi(H)\cap L)$.

For any group $G$ and a subgroup $H$, there is a natural map from
$\comm_G(H)$ to $\comm(H)$, taking $g\in \comm_G(H)$ to the element of
$\comm(H)$ represented by conjugation by $g$.  The kernel of
this homomorphism consists of the elements $g\in G$ that centralize some
finite-index subgroup of $H$.

It is easy to see that the abstract commensurator $\comm(\Z^n)$ can be
naturally identified with $GL(n,\Q)$.  Equivalently, if $L$ is a finitely
generated free abelian group then $\comm(L)$ is identified with
$GL(L\otimes\Q)$, the group of vector space automorphisms of $L\otimes\Q$.
There is a coordinate-free description of the homomorphism from the
abstract commensurator of $L$ to $GL(L\otimes \Q)$: suppose that
$\phi:L'\rightarrow L''$ is an isomorphism between finite-index subgroups
of $L$, and let $i:L'\rightarrow L$ and $j:L''\rightarrow L$ be the
inclusions.  Each of $i\otimes 1:L'\otimes \Q\rightarrow L\otimes \Q$
and $j\otimes 1:L''\otimes \Q \rightarrow L\otimes \Q$ is an
isomorphism.  The image $\tildephi$ of $(\phi:L'\rightarrow L'')$ in $GL(L\otimes \Q)$
is $\tildephi\coloneq (j\otimes 1)\circ(\phi\otimes 1)\circ(i\otimes 1)^{-1}:L\otimes \Q
\rightarrow L\otimes \Q$.

\subsection{Automatic structures and the fellow traveller property}
In this subsection we will briefly discuss the notions of automatic and biautomatic structures on groups. The reader is referred to
\cite[Section 2.3, 2.5]{Eps} for more details and examples.

Let $\cA$ be a finite set and let $G$ be a group with a map $\mu:\cA \to
G$. We will say that \emph{$G$ is generated by $\cA$} if the extension
of $\mu$ to a homomorphism from the free monoid $\cA^*$ to $G$ is
surjective. Elements of $\cA^*$ will be called \emph{words}, and if $W
\in \cA^*$ and $g \in G$ are such that $\mu(W)=g$, we will say that
\emph{$W$ represents $g$ in $G$}.
Given a word $W$ in $\cA^*$, $|W|$ will denote its length.
We will always assume that $\cA$ is \emph{closed under inversion}, that is there is an involution $\iota: \cA \to  \cA$,
where, for each $a \in A$, $\iota(a)$ is denoted $a^{-1}$ and satisfies $\mu(a^{-1})=\mu(a)^{-1}$ in $G$. Any subset $\cL \subseteq A^*$ will be called a
\emph{language} over $\cA$.

We can form the \emph{Cayley graph} $\Gamma(G,\cA)$, of $G$ with respect to
$\cA$ as follows: the vertices are elements of $G$ and for every $g \in G$ and $a \in A$ there is an edge
from $g$ to $g\mu(a)$, labelled by $a$.
Metrically, every edge in $\Gamma(G,\cA)$ will be considered as an isometric copy of the interval $[0,1]$.

We will use $\d_\cA(\cdot,\cdot)$ to denote the standard graph metric on
$\Gamma(G,\cA)$; its restriction to $G$ is the word metric corresponding to the generating set $\cA$.
For any element $g \in G$ we will use $|g|_\cA$ to denote
$\d_\cA(1_G,g)$; in other words, $|g|_\cA$ is the length of a shortest
word in $\cA^*$ representing $g$ in $G$. Note that $|g|_\cA=|g^{-1}|_\cA$ since $\cA$ is closed under inversion.

For an edge path $p$ in $\Gamma(G,\cA)$, $p_-$ and $p_+$ will denote the
\emph{start} and \emph{end} vertices of $p$ respectively, and $|p|$ will denote the
\emph{length} of $p$.  The \emph{label} of $p$ is a word from $\cA^*$ obtained by collating the labels of its edges.

Any edge path $p$ in $\Gamma(G,\cA)$ can be equipped
with the following \emph{ray parametrization}: $\widehat p:[0,\infty)
  \to \ga(G,\cA)$, where for each $t \in [0,|p|] \cap \Z$, $\widehat
  p(t)$ is the $t$-th vertex of $p$ (so that $\widehat p(0)=p_-$,
  $\widehat p(|p|)=p_+$), and $\widehat p(t)=p_+$ for all $t > |p|$;
  for every $s \in [0,|p|-1]\cap \Z$ and each $t \in [s,s+1]$,
  $\widehat p(t)$ is defined so that the restriction of $\widehat p$
  to $[s,s+1]$ is an isometry with the corresponding edge of $p$ in
  $\Gamma(G,\cA)$.

\begin{defn} Let $p, q$ be two edge paths in $\Gamma(G,\cA)$, with ray
  parametrizations $\widehat p, \widehat q :[0,\infty) \to \ga(G,\cA)$
    respectively, and let $\zeta \ge 0$ be a constant.
We will say that $p$ \emph{$\zeta$-follows} $q$ if $\d_\cA(\widehat
p(t),\widehat q(t)) \le \zeta$ for all $t \in [0,|p|]\cap \Z$.
The paths $p$ and $q$ are said to \emph{$\zeta$-fellow travel } if  $\d_\cA(\widehat p(t),\widehat q(t)) \le \zeta$ for all $t \in \Z$.

If $U$ and $V$ are two words from $\cA^*$, we can consider two edge
paths $p$, $q$ in $\ga(G,\cA)$ which start at $1_G$ and are labelled by $U$,
$V$ respectively.  We will say that $U$ \emph{$\zeta$-follows} $V$ if
the path $p$ $\zeta$-follows the path $q$. Similarly, $U$ and $V$ are said to \emph{$\zeta$-fellow travel} if $p$ and $q$ $\zeta$-fellow travel.
\end{defn}

It is easy to see that two edge paths (words) $\zeta$-fellow travel if and only if each of them $\zeta$-follows the other one.

\begin{defn}\label{def:biaut_struc} Let $G$ be a group. An \emph{automatic structure} on the group $G$ is a pair $(\cA,\cL)$,
where $\cA$ is a finite generating set of $G$, which comes equipped with a map
$\mu:\cA \to G$ as above and which is closed under inversion, and $\cL \subseteq \cA^*$ is a language satisfying the following conditions:

\begin{itemize}
  \item[(i)] $\mu(\cL)=G$;
  \item[(ii)] $\cL$ is a \emph{regular language}, i.e., $\cL$ is the accepted language of a finite state automaton $\mathfrak A$ over $\cA$;
  \item[(iii)] there exists $\zeta \ge 0$ such that for any two edge paths $p,q$ in $\Gamma(G,\cA)$, labelled by some words from $\cL$ and satisfying $p_-=q_-$ and
  $\d_\cA(p_+,q_+) \le 1$, $p$ and $q$ $\zeta$-fellow travel.
\end{itemize}
$(\cA,\cL)$ is a \emph{biautomatic structure} on $G$, if $\cL$ satisfies the conditions (i), (ii) and
\begin{itemize}
  \item [(iii')] there exists $\zeta \ge 0$ such that for any two edge paths $p,q$ in $\Gamma(G,\cA)$, labelled by some words from $\cL$ and satisfying
  $\d_\cA(p_-,q_-)\le 1$ and  $\d_\cA(p_+,q_+) \le 1$, $p$ and $q$ $\zeta$-fellow travel.
\end{itemize}
The group $G$ is said to be \emph{automatic} (\emph{biautomatic}) if it admits an automatic (respectively, biautomatic) structure.
\end{defn}

Obviously  condition (iii') is stronger than condition (iii), so every biautomatic structure on a group is also an automatic structure. Also,  condition (iii)
implies that for any two paths $p,q$ that are labelled by some words from $\cL$ and satisfy $p_-=q_-$ and $\d_\cA(p_+,q_+) \le C$, for some $C \in \N \cup \{0\}$,
$p$ and $q$ $(\zeta \max\{C,1\})$-fellow travel  in $\Gamma(G,\cA)$. And, if  (iii') holds, then the requirement $p_-=q_-$ can be relaxed to $\d_\cA(p_-,q_-) \le C$.

An automatic structure $(\cA,\cL)$ on a group $G$ is said to be
\emph{finite-to-one} if $|\mu^{-1}(g) \cap \cL|<\infty$ for every $g
\in G$. It is known (see \cite[Theorem 2.5.1]{Eps}) that every
automatic (biautomatic) structure can be refined to a finite-to-one
automatic (respectively, biautomatic) structure, hence from now on we
will assume that all the automatic and biautomatic structures are
finite-to-one. Without loss of generality we will also suppose that all the automata in this paper have no dead states.

It is worth noting that the nowadays standard definition of a biautomatic structure that we give above is not the same as the definition from \cite[Definition 2.5.4]{Eps}: see \cite{Amrhein}. However, for finite-to-one structures these definitions are equivalent \cite[Theorem 6]{Amrhein}.

\subsection{The boundary of an automatic structure}
\begin{defn}\label{def:tend_to_inf} Let $(\cA,\cL)$ be an automatic
  structure on $G$, and let $(W_i)_{i \in \N}$ be a sequence of words
  from $\cL$. We will say that this sequence
\emph{tends to infinity} if $|W_i|\to \infty$ as $i \to \infty$ and
there exists $\zeta \ge 0$ such that for any $i,j \in \N$, $W_i$
$\zeta$-follows $W_j$ whenever $i \le j$.

Suppose that $(U_i)_{i\in \N}$ and $(V_i)_{i \in \N}$ are two sequences of words
from $\cL$ tending to infinity, and, for each $i \in \N$, $p_i$,
$q_i$ are the edge paths in $\ga(G,\cA)$ starting at $1_G$ and labelled
by $U_i$, $V_i$ respectively. We will say that $(U_i)_{i\in \N}$ is
\emph{equivalent} to $(V_i)_{i \in \N}$ if  the Hausdorff distance between the corresponding sequences of
paths $(p_i)_{i\in \N}$ and $(q_i)_{i \in \N}$ is at most $\eta$ in
$\ga(G,\cA)$. In other words, there must exist $\eta \ge 0$ such
that for all $i \in \N$, any vertex of $p_i$ is at most
$\eta$ away from a vertex of $q_j$, for some $j \in \N$, and
vice-versa.
\end{defn}

Definition \ref{def:tend_to_inf} immediately implies the following.

\begin{rem}\label{rem:subseq} If $(U_{i_j})_{j \in \N}$ is a
  subsequence of a sequence of words $(U_i)_{i \in \N}$ tending to
  infinity, then $(U_{i_j})_{j \in \N}$ also tends to infinity and is
  equivalent to $(U_i)_{i \in \N}$.
\end{rem}

\begin{defn}\label{def:boundary} If $(\cA,\cL)$ is an automatic structure on a group $G$,
  then the \emph{boundary}, $\partial \cL$, of this automatic structure
  is the set of equivalence classes of sequences of words from $\cL$
  tending to infinity. If $\alpha \in \partial\cL$ is the equivalence
  class of a sequence $(U_i)_{i \in \N}$, we will say that this
  sequence \emph{converges} to the boundary point $\alpha$.
\end{defn}

The first definition of a boundary of an automatic structure was given by Neumann and Shapiro in \cite[pp. 459-460]{Neu-Sha}. It is not difficult to see that
there is a natural bijection between their boundary and ours. However, our Definition \ref{def:boundary} is  better suited for constructing
the action of a group on the boundary of a biautomatic structure (see Section~\ref{sec:act_of_comm} below).

Given an automaton $\mathfrak A$, by a \emph{path} in this automaton
we will mean any directed path in the corresponding graph. If $x$ is a
state of $\mathfrak A$, a \emph{cycle based at $x$} is a closed path
starting and ending at the state $x$.  A path in $\mathfrak A$ is
\emph{simple} if it does not pass through the same state twice. A
cycle based at a state $x$ in $\mathfrak A$ is \emph{simple} if it
passes through $x$ exactly twice (at its beginning and at its end) and
does not pass through any other state more than once. A path or a
cycle is \emph{non-trivial} if it has at least one edge.

The following lemma is a straightforward consequence of the
definitions.

\begin{lemma}\label{lem:bd_pt_from_cycle} Suppose that $(\cA,\cL)$ is an
  automatic structure on a group $G$, $\mathfrak A$ is a finite state automaton
  for $\cL$ and $W$ is the label of
a non-trivial cycle $w$ in $\mathfrak A$ based at some state $x$.  Let $S$
and $T$ be words from $\cA^*$ labelling some paths $s$ and $t$ connecting
the initial state of $\mathfrak A$ with $x$ and $x$ with an accept
state of $\mathfrak{A}$ respectively. Then $(SW^iT)_{i\in\N}$ is a
sequence of words from $\cL$ tending to infinity.  In particular, if
$\cL$ is infinite then $\partial \cL \neq \emptyset$.
\end{lemma}

\begin{defn}\label{def:simple_bd_pt} In the notation of Lemma \ref{lem:bd_pt_from_cycle}, if
  the paths $s,t$ and the cycle $w$ are all simple,
the sequence of words $(SW^iT)_{i\in\N}$ will be called a \emph{simple sequence
  tending to infinity} and the corresponding point $\alpha \in
\partial\cL$ will be called a \emph{simple boundary point}.
\end{defn}

\section{The boundary of an automatic structure on an abelian group}\label{sec:bd-abelian}
\begin{thm}\label{thm:bd_of_aut_struc_on_ab} Let $H$ be an abelian
  group with an automatic structure $(\cB, \cM)$ and let $\mathfrak{A}$
  be a finite state automaton accepting the language $\cM$. Then
  every boundary point $\alpha \in \partial\cM$ is simple; in other
  words, every sequence of words from $\cM$ tending to infinity is
  equivalent to a simple such sequence.
\end{thm}

\begin{proof} Let $w$ be a simple cycle in $\mathfrak A$ based at a
  state $x$. Given a path $v$ in $\mathfrak A$, an \emph{occurrence}
  of $w$ in $v$ is a subpath of $v$ starting and ending at $x$ and
  traversing $w$ exactly once. We will use $\log_{w}(v)$ to denote the
  number of occurrences of $w$ in $v$.

Let $(U_i)_{i \in \N}$ be a sequence of words from $\cM$ tending to
infinity, and for each $i \in \N$ fix a path $u_i$ from the initial
state to an accept state in $\mathfrak{A}$, labelled by $U_i$. Since
$|u_i|=|U_i|\to \infty$ as $i \to \infty$ and the automaton
$\mathfrak{A}$ is finite, in view of Remark \ref{rem:subseq} we can
replace $(U_i)_{i \in \N}$ by a subsequence so that for all $i \in \N$
$u_i$ and $u_{i+1}$ share a common initial subpath $v_i$, of length $i$.

Evidently, since $|v_i|=i \to \infty$ as $i \to \infty$, there must
exist some non-trivial simple cycle $w$, based at a state $x$ in
$\mathfrak A$, such that
\begin{equation}\label{eq:limsup}
\limsup_{i \to \infty}\log_w(v_i)=\infty.
\end{equation}

Choose $I \in \N$ so that $\log_w(v_I)>0$, then there is a path $s$
from the initial state of $\mathfrak A$ to $x$ such that $s w$ is an
initial subpath of $v_I$ (and, hence, of $u_i$ for all $i \ge I$). We
can also choose a simple path $t$ from $x$ to an accept state of
$\mathfrak A$, and let $S$, $T$ and $W$ be the labels of $s$, $t$ and
$w$ respectively. We will now show that the sequence of words $(SW^iT)_{i \in   \N}$, which tends to infinity by Lemma \ref{lem:bd_pt_from_cycle},
is equivalent to the original sequence $(U_i)_{i \in \N}$.

Let $p_i$ and $q_i$ be the paths in $\ga(H,\cB)$ starting at $1_H$ and
labelled by the words $U_i$ and $SW^iT$ respectively, $i \in \N$.  For
any $i \in \N$, according to \eqref{eq:limsup}, there is $j \ge I$
such that $\log_w(u_j) \ge i$, hence $u_j$ is the concatenation $u_j=s
w s_{1} w s_{2} \dots w s_{i}$, where $s_{1},\dots, s_{i-1}$ are some
paths in $\mathfrak A$ from $x$ to itself, and $s_{i}$ is a path from
$x$ to an accept state of $\mathfrak A$. Let $S_{k}$ be the word in
$\cB^*$ labelling $s_{k}$, $k=1,\dots, i$.

Note that since $H$ is abelian, the word $U_j=SWS_{1} \dots W S_{i}$
represents the same element of $H$ as the word $R\coloneq
SW^iS_{1}\dots S_{i}$.  Moreover, clearly $R$ is accepted by
$\mathfrak A$ (it labels the path $sw^i s_1 \dots s_i$ in $\mathfrak A$), by
construction, hence $U_j$ and $R$ $\e$-fellow travel in $\ga(H,\cB)$,
where $\e \ge 0$ is the constant from the definition of the automatic
structure $(\cB,\cM)$. Therefore the word $SW^i$ $\e$-follows $U_j$ in
$\ga(H,\cB)$, which implies that every vertex of the path
$q_i$ lies $\eta_1$-close to a vertex of $p_j$, where  $\eta_1 \coloneq \e+|T|$.

To show the converse, let $l \in \N$ be arbitrary and set $i\coloneq
|U_l|$. By \eqref{eq:limsup} there exists $j \ge \max\{l,i,I\}$ such
that $u_j=s w s_{1} w s_{2} \dots w s_{i}$ as before. If $S_k$ is the
label of $s_k$, $k=1,\dots,i$, arguing as above we can conclude that
the words $U_j$ and $R\coloneq SW^iS_{1}\dots S_{i}$ $\e$-fellow
travel in $\ga(H,\cB)$.  Since $(U_m)_{m\in\N}$ tends to infinity, we
know that there is $\zeta \ge 0$, depending only on this sequence,
such that $U_l$ $\zeta$-follows $U_j$, as $l \le j$. Hence $U_l$
$\zeta$-follows the prefix of $U_j$, of length $|U_l|=i$, which, in
its own turn, $\e$-follows the prefix of $R$, of the same
length. Since $SW^i$ is a prefix of $R$ of length at least $i$, we can
conclude that $U_l$ $(\zeta+\e)$-follows the word $SW^i$ in
$\ga(H,\cB)$, hence it also $(\zeta+\e)$-follows the word
$SW^iT$. Therefore each vertex of the path $p_l$ is $\eta_2$-close to
a vertex of the path $q_i$ in $\ga(H,\cB)$, where $\eta_2\coloneq
\zeta+\e$.

Thus we have shown that the sequences of paths $(p_i)_{i \in \N}$ and
$(q_i)_{i \in \N}$ lie in $\eta
\coloneq\max\{\eta_1,\eta_2\}$-neighborhoods of each other in
$\ga(H,\cB)$, yielding that the sequences $(U_i)_{i \in \N}$ and
$(SW^iT)_{i\in\N}$ are equivalent, as claimed.

The proof of the theorem is not quite finished yet, as the path $s$,
labelled by $S$ in $\mathfrak A$, may not be simple. So, choose some
simple path $s'$, joining the initial state of $\mathfrak A$ with the
state $x$, at which the simple cycle $w$ is based, and let $S'$ be the
word labelling $s'$.  By Lemma \ref{lem:bd_pt_from_cycle}, $(S'W^iT)_{i
  \in \N}$ is a simple sequence tending to infinity, and we will
complete the proof by showing that this sequence is equivalent to
$(SW^iT)_{i\in\N}$ (and, hence to $(U_i)_{i \in \N}$, by
transitivity).

Let $q_i'$ be the path in $\ga(H,\cB)$ starting at $1_H$ and labelled by
the word $S' W^i T$, $i \in \N$. Since $H$ is abelian, for all $i \in
\N$ the word $(SW^iT)^{-1} S' W^i T$ represents the same element of
$H$ as the word $S^{-1} S'$. Therefore $\d_\cB((q_i)_+,(q_i')_+) \le
\theta$ for all $i \in \N$ , where $\theta \coloneq |S^{-1}S'|$.
Since the labels of $q_i$ and $q_i'$ are both in the language $\cM$,
these paths $\lambda$-fellow travel in $\ga(H,\cB)$ for each $i \in \N$,
where $\lambda \coloneq \e\max\{1,\theta\}$.  Hence the Hausdorff
distance between the sequences of paths $(q_i)_{i \in \N}$ and
$(q_i')_{i \in \N}$ is at most $\lambda$, which implies that the
sequences of words $(SW^iT)_{i\in\N}$ and $(S'W^iT)_{i\in\N}$ are
equivalent.  Thus the theorem is proved.
\end{proof}

Definition \ref{def:simple_bd_pt} implies that any automatic structure admits only finitely many simple sequences of words tending to infinity.
Therefore the following statement is a consequence of Theorem \ref{thm:bd_of_aut_struc_on_ab} (cf. \cite[Theorem~6.7]{Neu-Sha}).

\begin{cor}\label{cor:finite_boundary} If $(\cB,\cM)$ is an automatic structure on an abelian group then $\partial \cM$ is finite.
\end{cor}

\section{The action of the commensurator on the boundary of a
  quasiconvex subgroup}\label{sec:act_of_comm}
Let $G$ be a group equipped with a (finite-to-one) biautomatic structure $(\cA, \cL)$.
Recall that a subgroup $H\leqslant G$ is $\cL$-\emph{quasiconvex} if
there exists $\varkappa \ge 0$ such that for any path $p$ in
$\ga(G,\cA)$ starting at $1_G$, ending at some $h \in H$ and labelled by
a word $W \in \cL$, every vertex of $p$ lies in the
$\varkappa$-neighborhood of $H$ (see \cite[p. 129]{G-S}).  Given such a
quasiconvex subgroup, let $\cL' \coloneq \cL \cap \mu^{-1}(H)$ and define a finite subset $\cB$ of $H$ by
\begin{equation}\label{eq:B} \cB\coloneq \{g^{-1} \mu(a) g'
\mid a \in \cA,~g, g' \in G,~|g|_\cA,|g'|_\cA \le \varkappa\} \cap H.
\end{equation}
Note that $\cB=\cB^{-1}$ since $\cA$ is closed under inversion.
Recalling the construction from \cite[p. 138]{G-S}, given any word
$W=a_1\dots a_n \in \cL'$, where $a_1,\dots,a_n \in \cA$, the
quasiconvexity of $H$ implies that for each $i \in \{1,\dots,n\}$
there is $g_i \in G$ such that $|g_i|_\cA \le \varkappa$ and
$\mu(a_1\dots a_i)g_i \in H$. Clearly we can choose $g_n =1_G$ and let
$g_0\coloneq 1_G$ so that $\mu(W)=\prod_{i=1}^n g_{i-1}^{-1} \mu(a_i)
g_i$ in $G$. This allows us to re-write the words from $\cL'$ as words
from $\cB^*$, possibly in a non-unique fashion due to some freedom in
the choice of $g_i$. Let $\cM \subseteq \cB^*$ denote the resulting
language, consisting of words from $\cL'$ re-written as words in $\cB^*$
in such a way (with all possible $g_i$'s).

In \cite[Theorem 3.1]{G-S} Gersten and Short proved that $(\cB,\cM)$ is
a biautomatic structure on $H$. For our purposes it will be convenient
to modify the original biautomatic structure on $G$ as follows.  Let
$\cC \coloneq \cA \sqcup \cB$ be the abstract union of the finite sets $\cA$
and $\cB$. Then $\cN \coloneq \cL \cup \cM$ is a language in $\cC^*$.
Obviously $\cL$ and $\cM$ are still regular in $\cC^*$, hence $\cN$ is
also a regular language in $\cC^*$, as a union of regular languages (see
\cite[Lemma 1.4.1]{Eps}).

One can define a map $\nu:\cC \to G$, where $\nu|_\cA=\mu$ and $\nu|_\cB$ is
the identity map on $\cB$ (as $\cB \subseteq G$, by definition), and
extend it to a monoid homomorphism $\nu:\cC^* \to G$. Clearly $\cC$ is a
finite generating set of $G$.

Consider any path $p$ in $\ga(G,\cC)$, labelled by a word  $U \in \cM$. Let $f\coloneq p_- \in G$, then $p_+$ belongs to the coset $fH$.
By definition, $U$ is obtained from a word $W \in \cL'=\cL \cap \mu^{-1}(H) \subseteq \cL$
by applying the re-writing process above.
Let $\tilde p$ denote the path in $\ga(G,\cC)$ labelled by $W$ and starting at $f=p_-$. Then ${\tilde p}_+=p_+$ and $|\tilde p|=|p|$.
Moreover, the paths $p$ and $\tilde p$  $\varkappa$-fellow travel in $\ga(G,\cC)$, by construction (see Figure \ref{fig:two_paths}).

\begin{figure}[!ht]
  \begin{center}
   \includegraphics{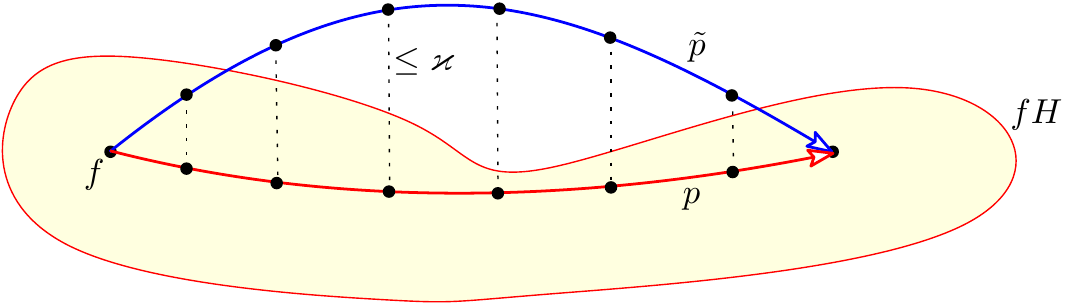}
  \end{center}
  \caption{The paths $p$ and $\tilde p$.}\label{fig:two_paths}
\end{figure}

\begin{lemma}\label{lem:new_biaut_struc} $(\cC,\cN)$ is a finite-to-one
  biautomatic structure on $G$.
\end{lemma}

\begin{proof} We have already observed that $\cN$ is a regular
  language in $\cC^*$; moreover, $\nu(\cN)=G$ as $\nu(\cL)=G$ and $\cL
  \subseteq \cN$.  Let us now check that the fellow travelling
  property holds.

Since $(\cA,\cL)$ is a biautomatic structure on $G$, there is some $\e
\ge 0$ such that two paths in $\ga(G,\cA)$ with labels from $\cL$ whose endpoints are at distance at most $1$ from each other
$\e$-fellow travel.  Consider any two
paths $p$ and $q$ in $\ga(G,\cC)$ labelled by words from $\cN$, with
$\d_\cC(p_-,q_-) \le 1$ and $\d_\cC(p_+,q_+) \le 1$.  Define the paths
$p'$ and $q'$ in $\ga(G,\cC)$ as follows. If $p$ is labelled by a word
from $\cL$, then $p'=p$; otherwise, if $p$ is labelled by a word from
$\cM$ then $p'$ is the path $\tilde p$ defined above, labelled by a
word from $\cL$. We construct the path $q'$ similarly, and note that
since the labels of $p'$ and $q'$ are in $\cA^*$, they can be considered
as paths in $\ga(G,\cA)$.

Since $|b|_\cA \le 2\varkappa+1$ for any $b \in \cB$, by \eqref{eq:B},
$p'$ has the same endpoints as $p$, and $q'$ has the same endpoints as
$q$, we see that $\d_\cA(p'_-,q'_-)\le 2\varkappa+1$ and
$\d_\cA(p'_+,q'_+) \le 2\varkappa+1$. Therefore the paths $p'$ and $q'$
$\zeta$-fellow travel in $\ga(G,\cA)$, where $\zeta\coloneq \e
(2\varkappa+1)$. Hence, these paths also $\zeta$-fellow travel in
$\ga(G,\cC)$, as $\cA \subseteq \cC$. It follows that the original paths $p$
and $q$ $\lambda$-fellow travel in $\ga(G,\cC)$, where $\lambda\coloneq
\zeta+2\varkappa$. Thus $(\cC,\cN)$ is a biautomatic structure on $G$.

By our convention, the biautomatic structure $(\cA,\cL)$ on $G$ is
finite-to-one, and, by the above construction, there are only finitely
many possibilities for re-writing each word $W \in \cL'$ as a word in
$\cB^*$, hence the biautomatic structure $(\cB,\cM)$ on $H$ is also
finite-to-one. It follows that the biautomatic structure $(\cC,\cN)$ on
$G$ is finite-to-one as well.
\end{proof}

The new biautomatic structure $(\cC, \cN)$ on $G$ naturally extends the
biautomatic structure $(\cB,\cM)$ on $H$, and the Cayley graph
$\ga(H,\cB)$ is a subgraph of the Cayley graph $\ga(G,\cC)$.  This allows
us to define the action of the commensurator $\comm_G(H)$ on the
boundary of $\cM$ as follows. Let $(U_i)_{i \in \N}$ be a sequence of
words from $\cM$ tending to infinity and representing a boundary point
$\alpha \in \partial \cM$, and let $g \in \comm_G(H)$.  Since $H \cap
gHg^{-1}$ has finite index in $gHg^{-1}$, there exist $f_1,\dots,f_k
\in G$ such that $gHg^{-1} \subseteq H f_1 \cup \dots \cup H f_k$,
thus, for any $gh \in gH$ there is $e \in G$, with $|e|_\cC \le
\max\{|f_j g|_\cC \mid j=1,\dots,k\}<\infty$, such that $ghe \in H$.

Note that $g\,\nu(U_i)\in gH$ for all $i \in \N$, and choose any
sequence of elements $e_i \in G$, $i \in \N$, such that
$\sup\{|e_i|_\cC\mid i \in \N\} <\infty$ and $g\, \nu(U_i)\, e_i \in H$
for all $i \in \N$. Let $V_i \in \cM$ be any word representing $g\,
\nu(U_i)\, e_i$, $i \in \N$. We claim that the sequence of words
$(V_i)_{i \in \N}$ tends to infinity and define its equivalence class
in $\partial \cM$ to be the result of the action of $g$ on $\alpha$.

\begin{lemma}\label{lem:action_of_comm_on_bd} The above construction gives a well-defined action of
  $\comm_G(H)$ on $\partial \cM$.
\end{lemma}

\begin{proof} Using the notation from the preceding paragraph, let us
  first check that the sequence of words $(V_i)_{i \in \N}$ in $\cM$
  tends to infinity.  Note that $\lim_{i\to \infty} |\nu(U_i)|_\cC =
  \infty$, as $\lim_{i \to \infty}|U_i|= \infty$, and the biautomatic
  structure $(\cC,\cN)$ is finite-to-one by Lemma
  \ref{lem:new_biaut_struc}. We can also observe that
\[|V_i| \ge |g\, \nu(U_i)\, e_i|_\cC \ge |\nu(U_i)|_\cC-|g|_\cC-|e_i|_\cC
\to \infty, \mbox{ as } i \to \infty.\]
Therefore $\lim_{i \to \infty} |V_i|=\infty$.

For each $i \in \N$ let $p_i$ be the path in $\ga(G,\cC)$ starting at
$g$ and labelled by $U_i$. Since the sequence $(U_i)_{i \in \N}$, of
words in $\cM \subseteq \cN$, tends to infinity, there exists $\zeta
\ge 0$ such that $p_i$ $\zeta$-follows $p_j$ in $\ga(G,\cC)$, for any $i,j \in \N$ with
$i \le j$.

Let $q_i$ be the path in $\ga(G,\cC)$ starting at $1_G$ and labelled by
$V_i$, $i \in \N$. Observe that for all $i \in \N$,
$\d_\cC((p_i)_-,(q_i)_-)=|g|_\cC$ and
$\d_\cC((p_i)_+,(q_i)_+)=\d_\cC(g\,\nu(U_i),g\,\nu(U_i)e_i)=|e_i|_\cC$. Since
$\sup\{|e_i|_\cC \mid i \in\N\} <\infty$ and the structure $(\cC,\cN)$ is
biautomatic, there exists $\eta \ge 0$
such that the paths $p_i$ and $q_i$ $\eta$-fellow travel, for all $i
\in \N$. Therefore, for $\lambda \coloneq 2\eta+\zeta$, the path $q_i$
$\lambda$-follows the path $q_j$ in $\ga(G,\cC)$, whenever $i,j \in \N$
and $i \le j$. Let $\lambda' \coloneq \max\{|h|_\cB \mid h \in H,~|h|_\cC
\le \lambda\} < \infty$.  Since each $q_i$ is labelled by a word from
$\cB^*$, $q_i$ $\lambda'$-follows $q_j$ as paths in $\ga(H,\cB)$, whenever
$i \le j$.  Thus the sequence $(V_i)_{i \in \N}$, of words in $\cM$,
tends to infinity.

Now, suppose that $(U_i')_{i \in \N}$ is another sequence of words in
$\cM$ which tends to infinity and is equivalent to the sequence
$(U_i)_{i \in \N}$. For each $i \in \N$ choose an arbitrary element
$e_i' \in G$ such that $\sup\{|e_i'|_\cC\mid i \in \N\} <\infty$ and
$g\, \nu(U_i') \, e_i' \in H$, and let $V_i'$ be any word from $\cM$
representing the element $g\, \nu(U_i') \, e_i'$. Let us show that the
sequence $(V_i')_{i \in \N}$ is equivalent to the sequence $(V_i)_{i  \in \N}$.

Let $p_i'$ be the path in $\ga(G,\cC)$ starting at $g$ and labelled by
$U_i'$, and let $q_i'$ be the path in $\ga(G,\cC)$ starting at $1_G$ and
labelled by $V_i'$, $i \in \N$.  By the assumptions, there exists
$\theta \ge 0$ such that the Hausdorff distance between the sequences
$(p_i)_{i \in \N}$ and $(p_i')_{i \in \N}$ is at most $\theta$. On the
other hand, the argument above shows that the paths $p_i'$ and $q_i'$
$\eta'$-fellow travel, for some $\eta' \ge 0$ and all $i \in \N$.  It
follows that the Hausdorff distance between the sequences $(q_i)_{i
  \in \N}$ and $(q_i')_{i \in \N}$ does not exceed $\theta+\eta+\eta'$
in $\ga(G,\cC)$. Since $q_i$ and $q_i'$ are labelled by words from $\cB^*$,
for each $i \in \N$, the Hausdorff distance between the sequences of
these paths is also bounded in $\ga(H,\cB)$ by the constant $\max\{|h|_\cB
\mid h \in H,~|h|_\cC \le \theta+\eta+\eta'\}$. Thus the sequences of
words $(V_i)_{i \in \N}$ and $(V_i')_{i \in \N}$ indeed give rise to
the same boundary point in $\partial \cM$. This shows that the above
action is well-defined.

It remains to check that the axioms of a group action are
satisfied. Let $g,g' \in \comm_G(H)$ and let $(U_i)_{i\in \N}$ be a
sequence of words from $\cM$ converging to a boundary point $\alpha
\in \partial\cM$.  Obviously $1_G\,\alpha=\alpha$, so this axiom is
satisfied. On the other hand, by the definition of the action, the point
$g'(g \, \alpha) \in \partial \cM$ is obtained from a sequence of
words in $\cM$ representing the elements $g' g \, \nu(U_i)\, e_i
e_i'$, where $e_i,e_i' \in G$, $i \in \N$, $\sup\{|e_i|_\cC,|e_i'|_\cC
\mid i \in \N\}<\infty$ and $g\, \nu(U_i)\, e_i \in H$, $g' g \,
\nu(U_i)\, e_i e_i' \in H$, for all $i \in \N$. Set $e_i''\coloneq e_i
e_i' \in G$, and observe that since $|e_i''|_\cC \le |e_i|_\cC+|e_i'|_\cC$
for all $i \in \N$, $\sup\{|e_i''|_\cC \mid i \in
\N\}<\infty$. Therefore the boundary point $(g'g)\, \alpha$ can be
obtained from the same sequence of words in $\cM$, representing the
same elements $g' g \, \nu(U_i)\, e_i''$, $i \in \N$, that were used
for the point $g'(g \, \alpha)$. Thus $(g'g)\, \alpha=g'(g\, \alpha)$,
which completes the proof of the lemma.
\end{proof}

\section{The case of an abelian subgroup}\label{sec:abgp}
In this section we will assume that $G$ is a group equipped with a
biautomatic structure $(\cC, \cN)$, and $H\leqslant G$ is a finitely generated abelian
subgroup with a biautomatic structure $(\cB, \cM)$, where
$\cB \subseteq
\cC$ and $\cM=\cN \cap \cB^*$. As explained in Section
\ref{sec:act_of_comm}, we can find such biautomatic structures on $G$
and $H$ starting from any biautomatic structure $(\cA,\cL)$ on $G$, as
long as $H$ is $\cL$-quasiconvex. We also define the action of the
commensurator $\comm_G(H)$ on the boundary $\partial\cM$ as explained
in that section. Let $\mu:\cB^* \to H$ and $\nu:\cC^*\to G$ denote the
monoid homomorphisms sending the words to the group elements they
represent.

The next theorem is the main result of this section.

\begin{thm}\label{thm:one_elt_fixing_bd} Using the notation from the
  beginning of the section, suppose that an element $g \in \comm_G(H)$
  acts trivially on the boundary $\partial\cM$. Then $g$ centralizes
  a finite-index subgroup of $H$.
\end{thm}

The proof of Theorem \ref{thm:one_elt_fixing_bd} will require two
auxiliary statements.

\begin{lemma}\label{lem:fix_bd_pt->centr_power} Suppose that $g \in
  \comm_G(H)$ fixes a simple boundary point $\alpha \in \partial \cM$,
  given by a sequence of words $(SW^iT)_{i \in \N}$ in $\cM$ tending
  to infinity, where $S,W,T \in \cB^*$. Then there exists $m \in \N$
  such that $g h^m g^{-1}=h^m$ in $G$, where $h \in H$ is the element
  represented by the word $W$.
\end{lemma}

\begin{proof} Since $g \in \comm_G(H)$, we can choose elements $e_i \in
  G$, $i \in \N$, so that $\eta \coloneq \sup\{|e_i|_\cC \mid i \in \N\} < \infty$ and $g\, \nu(S  W^iT)\, e_i \in H $ for all $i \in \N$.  By the definition of the
  action, given in Section~\ref{sec:act_of_comm}, $g\,\alpha$ is the
  equivalence class of a sequence of words $(V_i)_{i \in \N}$, where
  $V_i \in \cM $ represents the element $g\, \nu(S W^iT)\, e_i$, $i
  \in \N$.

For each $i \in \N$ let $p_i$ be the path in $\ga(G,\cC)$ starting at
$g$ and labelled by the word $SW^iT$, let $q_i$ and $r_i$ be the paths
in $\ga(G,\cC)$ starting at $1_G$ and labelled by the words $V_i$ and
$SW^iT$, respectively (see Figure \ref{fig:1}). Note that
$\d_\cC((p_i)_+,(q_i)_+)=|e_i|_\cC\le \eta$ and, since $g\,\alpha=\alpha$,
there are $\theta \ge 0$, independent of $i$, and $j=j(i) \in \N$,
such that $(q_i)_+$ is at most $\theta$-away from some vertex of
$r_j$. Obviously, any vertex of $r_j$ is at most $\varkappa$-away from
$(r_k)_+$, for some $k \le j$ and $\varkappa \coloneq |S|+|W|+|T|$.

\begin{figure}[!ht]
  \begin{center}
   \includegraphics{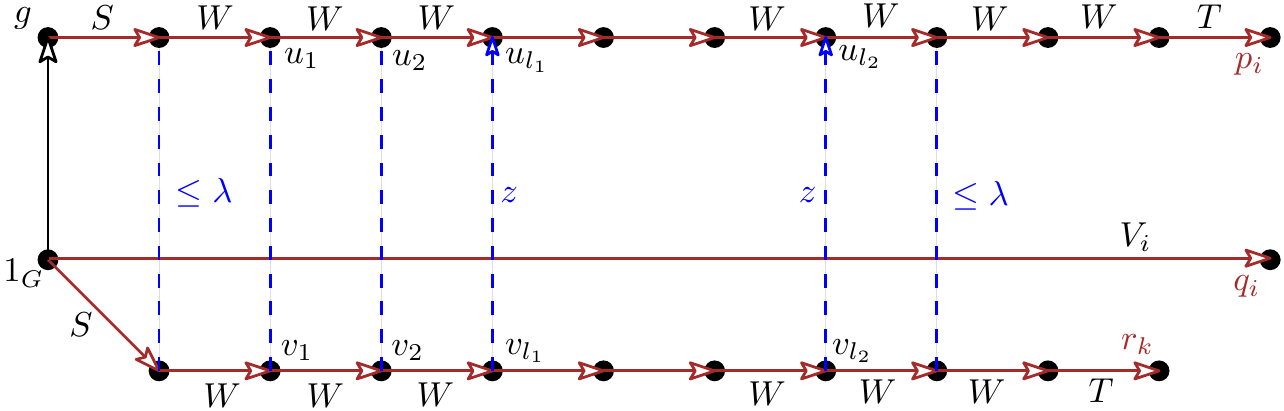}
  \end{center}
  \caption{}\label{fig:1}
\end{figure}

Thus for every $i \in \N$ there is $k=k(i) \in \N$ such that
$\d_\cC((p_i)_+),(r_k)_+) \le \eta+\theta+\varkappa$. On the other hand,
recall that $\d_\cC((p_i)_-,(r_k)_-)=\d_\cC(g,1_G)=|g|_\cC$. Since the paths
$p_i$ and $r_k$ are both labelled by words from $\cM\subseteq \cN$ and
$(\cC,\cN)$ is a biautomatic structure on $G$, there is $\lambda \ge 0$
such that these paths $\lambda$-fellow travel in $\ga(G,\cC)$ for all $i
\in \N$ and $k=k(i)$.

Now, the word $W$ cannot be empty by the assumptions, and since the
structure $(\cC,\cN)$ is finite-to-one, $\nu(W)=h$ must have infinite
order in $G$. It follows that $\lim_{i \to \infty}
\d_\cC((p_i)_-,(p_i)_+)=\infty$, hence $k(i)$ must also tend to infinity
as $i \to \infty$. Therefore there is $i_0 \in \N$ such that both
$i_0$ and $k_0\coloneq k(i_0)$ are greater than $\xi \coloneq |\{f \in
G \mid |f|_\cC \le \lambda\}|+1$.

For each $l \in \{1,\dots,\xi\}$ let $u_l \in G$ be the vertex of
$p_{i_0}$ such that the subpath of $p_{i_0}$ from $g=(p_{i_0})_-$ to
$u_l$ is labelled by $SW^l$.  Similarly, we define $v_l\in G$ to be the
vertex of the path $r_{k_0}$ such that the subpath of $r_{k_0}$ from
$1_G=(r_{k_0})_-$ to $v_l$ is labelled by $SW^l$.  By construction,
$|v_l^{-1}u_l|_\cC=\d_\cC(u_l,v_l) \le \lambda$ for every $l=1,\dots,\xi$,
and the definition of $\xi$ implies that there must exist indices
$l_1,l_2$, $1 \le l_1 < l_2 \le \xi$, such that
$v_{l_1}^{-1}u_{l_1}=v_{l_2}^{-1}u_{l_2}$ in $G$.  Set $z\coloneq
v_{l_1}^{-1}u_{l_1}$, then the quadrilateral in $\ga(G,\cC)$ with
vertices $v_{l_1}$, $u_{l_1}$, $u_{l_2}$ and $v_{l_2}$ (see Figure
\ref{fig:1}) gives rise to the equality $z\, \nu(W^m)\, z^{-1}
=\nu(W^m)$ in $G$, where $m \coloneq l_2-l_1 \in \N$. Thus $z$
commutes with $h^m$ in $G$. On the other hand, the quadrilateral with
vertices $1_G$, $g$, $u_{l_1}$ and $v_{l_1}$ in $\ga(G,\cC)$ gives rise
to the equality $g= \nu(SW^{l_1})\, z \,\nu(SW^{l_1})^{-1}$ in
$G$. Since $SW^{l_1}$ is a word from $\cB^*$ and $H$ is abelian, the
element $\nu(SW^{l_1})$ commutes with $h^m$ in $G$, therefore $g$ also
commutes with $h^m$ and the lemma is proved.
\end{proof}

\begin{lemma}\label{lem:simple_cycles_gen_a_fi_sbgp} Let $\mathfrak A$
  be a finite state automaton accepting the language $\cM$, and let
  $D$ be the set of elements of $H$ represented by the labels of
  non-trivial simple cycles in $\mathfrak A$. Then $D$ generates a
  finite-index subgroup of $H$.
\end{lemma}

\begin{proof} Let $n$ be the number of states in the automaton
  $\mathfrak A$. We claim that $H \subseteq E \langle D \rangle$,
  where $E \coloneq \{h \in H \mid |h|_\cB \le n\}$.

Indeed, choose any element $f \in H$ and let $V$ be the shortest word
from $\cM$ representing $f$ in $H$. We will prove that $f \in E\langle
D \rangle$ by induction on the length of $V$.  If $|V| \le n$ then $f
\in E$. Otherwise, if $|V|>n$, any path $v$ in $\mathfrak{A}$ labelled
by $V$, from the initial state to an accept state of $\mathfrak A$
(which exists as this automaton accepts $\cM$), must contain a
non-trivial simple cycle based at some state $x$ and labelled by a word
$W$, with $\mu(W) \in D$.  Thus $V=S W T$, where $S$, $T$ are labels
of some subpaths of $v$ ending and beginning at $x$ respectively.
Clearly $ST \in \cM$, as this word is accepted by $\mathfrak A$, but
its length is strictly smaller than the length of $V$. Moreover, since
$H$ is abelian, we have $f=\mu(V)=\mu(ST)\mu(W) \in \mu(ST)D$.  By the
induction hypothesis, $\mu(ST) \in E\langle D \rangle $, so $f \in E
\langle D \rangle D=E\langle D \rangle$, and the claim is proved.

Since $E$ is a finite set, by definition, and $\langle D \rangle
\leqslant H$, the inclusion $H \subseteq E\langle D \rangle$ implies
that $|H:\langle D \rangle|<\infty$, as required.
\end{proof}

\begin{proof}[Proof of Theorem \ref{thm:one_elt_fixing_bd}] Let
  $\mathfrak A$ be a finite state automaton accepting the language
  $\cM$, and let $\mathcal{W}=\{W_1,\dots,W_n\}$ be the list of the
  labels of all non-trivial simple cycles in $\mathfrak A$. Let $D$
  denote the (finite) set of elements of $H$ represented by the words
  from $\mathcal W$.

Take any $h \in D$, then $h=\mu(W_k)$ for some $k \in
\{1,\dots,n\}$. Let $x$ be the state of $\mathfrak A$ at which a cycle
labelled by $W_k$ is based, and choose some simple paths $S$ and $T$
joining the initial state of $\mathfrak A$ with $x$ and $x$ with an
accept state of $\mathfrak A$ respectively.  Then, according to
Lemma~\ref{lem:bd_pt_from_cycle}, the sequence of words $(SW_k^i T)_{i
  \in \N}$ converges to some point $\alpha \in \partial \cM$. By the
assumptions, $g\,\alpha=\alpha$, so we can use Lemma
\ref{lem:fix_bd_pt->centr_power} to conclude that $gh^m g^{-1}=h^m$ in
$G$, for some $m=m(h) \in \N$. Since the latter holds for every $h \in
D$ and $|D|<\infty$, we can find a single $l \in \N$ such that $g$
commutes with $h^l$ for all $h \in D$.

Now, the elements $\{h^l \mid h \in D\} $ obviously generate a
finite-index subgroup $H'$ of the finitely generated abelian group $\langle D
\rangle$, which itself has finite index in $H$, by Lemma
\ref{lem:simple_cycles_gen_a_fi_sbgp}. Thus $|H:H'|< \infty$ and $H'$
is centralized by $g$, as required.
\end{proof}

We can now prove Theorem \ref{thm:comm_of_ab_qc_sbgp} stated in the introduction.
\begin{proof}[Proof of Theorem \ref{thm:comm_of_ab_qc_sbgp}]
Define a biautomatic structure $(\cB,\cM)$ on the subgroup $L$ as in Section~\ref{sec:act_of_comm}.
This gives rise to an action of $\comm_G(H)$ on the finite set $\partial \cM$
(see Lemma~\ref{lem:action_of_comm_on_bd} and Corollary~\ref{cor:finite_boundary}), and we denote the kernel of this action by $\comm_G^0(H)$. Then $\comm_G^0 (H)\lhd\comm_G(H)$ and $|\comm_G(H):\comm_G^0(H)|<\infty$.

By Theorem~\ref{thm:one_elt_fixing_bd}, every element of $\comm_G^0(H)$ centralizes a finite-index subgroup of $H$, hence $\comm_G^0(H)$ lies in the kernel of the homomorphism from $\comm_G(H)$ to $\comm(H)$. It follows that the image of $\comm_G(H)$ in $\comm(H)$ is finite.
Any finite subset of $\comm_G^0(H)$ centralizes a finite-index subgroup of $H$, thus the same holds for any finitely generated subgroup $F \leqslant \comm_G^0(H)$.
\end{proof}

The main examples of $\cL$-quasiconvex subgroups in biautomatic groups are centralizers of finite subsets (see \cite[Corollary~8.3.5 and Theorem~8.3.1]{Eps}). Therefore the
following statement is an immediate corollary of Theorem~\ref{thm:comm_of_ab_qc_sbgp}.

\begin{cor}\label{cor:comm_of_centralizer} Let $G$ be a biautomatic group and let $X \subseteq G$ be a finite subset such
that $H \coloneq\mathrm{C}_G(X)$ is abelian. Then there is a finite-index subgroup
 $\comm_G^0(H) \lhd\comm_G(H)$ such that every finitely generated subgroup of $\comm_G^0(H)$
centralizes a finite-index subgroup of $H$ in $G$.
\end{cor}

\begin{rem} In \cite[Proposition~9.1]{huangprytula} Huang and Prytu\l{a} use an example from Wise's thesis \cite{Wise-thesis} to show that there exists a group $G$, acting properly discontinuously, cocompactly and cellularly on a product of two trees, and an infinite cyclic subgroup $H \leqslant G$ such that $\comm_G(H)$ is not finitely generated and does not normalize any finite-index subgroup of $H$.

Since the product of two trees is a \cat{} cube complex, the group $G$ is biautomatic by \cite{nibree}. After analysing the construction it
becomes clear that one can replace $H$ with a commensurable infinite cyclic subgroup to ensure that $H=\mathrm{C}_G(X)$ for some finite subset $X \subseteq G$. Therefore, the examples of  $G$ and $H$ show that it is indeed necessary to pass to finitely generated subgroups of $\comm_G(H)$  in Theorem~\ref{thm:comm_of_ab_qc_sbgp} and Corollary~\ref{cor:comm_of_centralizer}.
\end{rem}

\section{Commensuration and the Flat Torus Theorem}\label{sec:flat_torus}
Let us start this section by recalling the Flat Torus Theorem~\cite[II.7.1]{brihae}.  Throughout this section
we will use additive notation for the group operation on
a free abelian group $L$.

\begin{thm} \label{thm:flattorus}
Let $L$ be a free abelian group of rank $n$ acting properly by
semi-simple isometries on a \cat{} space $X$.  Then:

\begin{enumerate}

\item{} The min set $M$ for $L$ is non-empty and $M=Y\times \E^n$.

\item{} Every $c\in L$ leaves $M$ invariant, respects the product
decomposition, and acts trivially on $Y$ and by translation on
$\E^n$.

\item{} For $y\in Y$, the quotient $(\{y\}\times \E^n)/L$ is an
  $n$-torus.

\item{} If an isometry of $X$ normalizes $L$ then it preserves $M$
and the direct product decomposition.

\item{} If a group $\Gamma$ of isometries of $X$ normalizes $L$, then
a finite-index subgroup of $\Gamma$ centralizes $L$.  If $\Gamma$ is
finitely generated, then $\Gamma L$ has a finite-index subgroup
containing $L$ as a direct factor.

\end{enumerate}

\end{thm}

We want an analogous statement to (4) above, but for isometries that
lie in the commensurator of $L$ rather than in its normalizer.  For
this it is easier first to describe a different statement that is
equivalent to~(3).

Recall that a torsor for an abelian group is a non-empty set on which
it acts freely and transitively.  An affine space is naturally a torsor
for its vector space of translations.

\begin{rem}\label{rem:equiv}
Let $L$ be a free abelian group of finite rank $n$, and suppose that
$L$ acts by translations on a finite-dimensional real affine space
$\A$.  The following are equivalent:
\begin{itemize}
\item the action of $L$ is  properly discontinuous and   cocompact;

\item the unique affine extension of the action to $L\otimes \R$
makes $\A$ a torsor for $L\otimes \R$.
\end{itemize}
\end{rem}

In fact, the affine extension to $L\otimes \R$ is free if and only
if the action of $L$ on $\A$ is  properly discontinuous, and in
this case $\A$ is a torsor for $L\otimes \R$ if and only if $\A$ has
dimension~$n$.

Now suppose that one is given a torsor action of a vector space $V$ by
translations on the Euclidean space $\E^n$.  In this case, the
Euclidean distance on $\E^n$ enables one to define an inner product on
$V$, via $\displaystyle \langle v,w\rangle\coloneq\frac12 \bigl(d((v+w)\,x,x)^2- d(v\,x,x)^2-d(w\,x,x)^2\bigr)$, for any $x\in \E^n$.

In particular, with hypotheses and notation as in
Theorem~\ref{thm:flattorus}, we may define an inner product
$\langle \,\cdot\,,\,\cdot\,\rangle_{L}$ on $L\otimes \R$ by
setting
\begin{equation}\label{eq:inner_prod_formula}
\langle b,c\rangle_{L}\coloneq \frac{1}{2}\left(d((b+c)\,x,x)^2
-d(x,b\,x)^2-d(x,c\,x)^2 \right)
\end{equation}
for each $b,c\in L$, and extending linearly to $L\otimes \R$.
Here $x\in M$, the min set of $L$, and Theorem~\ref{thm:flattorus}(2)
tells us that the definition does not depend on which $x$ we choose.
The fact that this is an inner product follows easily from the cosine rule.
The following observation is an immediate consequence of the definitions and Theorem~\ref{thm:flattorus}.

\begin{rem}\label{rem:sp_on_sbgp} Let $L$ be a finitely generated free abelian group acting properly by isometries on a \cat{} space $X$, and let $L' \leqslant L$ be any subgroup. Then the min set of $L$ is contained in the min set of $L'$, and for all $ b,c \in L'$ we have $\langle b,c\rangle_{L'}=\langle b,c\rangle_{L}$.
\end{rem}

We are now ready state our addendum to the Flat Torus Theorem.

\begin{thm}\label{thm:addendum}
Let $L$ be a free abelian group of rank $n$ acting properly by
semi-simple isometries on a \cat{} space $X$.  Then:

\begin{itemize}

\item[(1)]{} The min set $M$ for $L$ is non-empty and $M=Y\times \E^n$.

\item[(2)]{} Every $c\in L$ leaves $M$ invariant, respects the product
decomposition, and acts trivially on $Y$ and by translation on
$\E^n$.

\item[($3'$)]{} For each $y\in Y$, $\{y\}\times \E^n$ is a torsor for
$L\otimes \R$ under the affine extension of the action of $L$.

\item[($4'$)]{} For any isometry $\varphi$ of $X$ that commensurates $L$, the
  image of $\varphi$ in $GL(L\otimes\Q)\leqslant GL(L\otimes\R)$
  preserves the inner product $\langle\cdot,\cdot\rangle_L$.

\end{itemize}
\end{thm}

\begin{proof}
Statements (1) and (2) are parts of the usual Flat Torus Theorem
(Theorem~\ref{thm:flattorus}), and are restated here for convenience.
By Remark~\ref{rem:equiv}, ($3'$) is equivalent to
Theorem~\ref{thm:flattorus}(3).  It remains to establish ($4'$).

Since $\varphi$ commensurates $L$, conjugation by it in the group of the isometries of $X$ induces
an isomorphism $\phi:L'\rightarrow L''$, for some finite-index subgroups $L'$ and $L''$ of $L$.  Let $M'$ and
$M''$ be the min sets in $X$ for $L'$ and $L''$ respectively.
Note that $M$, the min set for $L$, is contained in both $M'$
and $M''$, and that $\varphi$ restricts to an isometry from $M'$ to $M''$.
It follows that $\phi:L'\rightarrow L''$ respects their
inner products, in the sense that for each $b,c\in L'$,
$\langle \phi(b),\phi(c)\rangle_{L''}=\langle b,c\rangle_{L'}$.
Since $L'$ and $L''$ have finite index in $L$, there exists
$m \in \N$ so that for all $b\in L$, $mb\in L',L''$.
Remark~\ref{rem:sp_on_sbgp} implies that
$\langle \,\cdot\,,\,\cdot\,\rangle_{L}$,
$\langle \,\cdot\,,\,\cdot\,\rangle_{L'}$ and
$\langle \,\cdot\,,\,\cdot\,\rangle_{L''}$ are all equal on the finite-index subgroup $L' \cap L''$ of $L$.
Hence for any $b,c\in L$,
\[\langle \tildephi(b),\tildephi(c) \rangle_L = \frac{1}{m^2}\langle \phi(mb),\phi(mc) \rangle_{L''}=
\frac{1}{m^2}\langle mb,mc\rangle_{L'}=
\langle b,c\rangle_L,\]
where $\tildephi$ is the image of $\phi$ in $GL(L\otimes \Q)$, as
defined in Subsection~\ref{subsec:comm}.
\end{proof}

\section{Commensurating HNN-extensions of free abelian groups}\label{sec:comm_HNN}

Let $L$ be a finitely generated free abelian group, and let $\phi:L'
\rightarrow L''$ be an isomorphism between finite-index subgroups of $L$.
Define a group $G(L,\phi,L')$ as the HNN-extension of $L$ in which the
stable letter conjugates $L'$ to $L''$ via $\phi$:
\begin{equation*}
G(L,\phi,L')\coloneq \langle L, t\,\| \, tct^{-1}=\phi( c),~\forall\, c\in L'
\rangle.
\end{equation*}
In the case when we are given a basis for $L\cong\Z^n$ and $\phi$
is described by a matrix, we simplify the notation slightly.
For $A\in GL(n,\Q)$ and $L'$ a finite-index subgroup of $L\cap A^{-1}\,L=
\Z^n\cap A^{-1}\,\Z^n$, we write $G(A,L')$ for the HNN-extension defined
as above:
\begin{equation}\label{eq:pres_of_G(A,L')}
G(A,L')\coloneq \langle L, t\,\| \, tct^{-1}= A\,c,~\forall\, c\in L' \rangle.
\end{equation}

If in this case, $L'$ is as large as possible, i.e., $L'=L\cap A^{-1}\,L$,
then we write $G(A)$ instead of $G(A,L')$.

When $n=1$, the groups $G(A,L')$ are precisely the
Baumslag-Solitar groups;
if $A$ is a $1\times 1$ matrix with entry $m/d$, then $G(A,d\Z)=BS(m,d)$
and if $(m,d)=1$ then $G(A,d\Z)=G(A)$.

\begin{prop}\label{prop:free-by-ab-by-cyc}
Each group $G=G(A,L')$ is free-by-abelian-by-cyclic.
\end{prop}

\begin{proof}
There is an affine action of $G$ on $L\otimes \R\cong \R^n$
in which elements of $L$ act as translations and $t$ acts
as multiplication by the matrix $A$.  Let $\alpha:G \to AGL(n,\R)$ denote the resulting homomorphism, where
$AGL(n,\R)$ is the group of affine transformations of $\R^n$.
The subgroup $L$ is in the kernel of the standard map $\beta: AGL(n,\R) \to GL(n,\R)$, and hence the image $(\beta \circ \alpha)(G)$ is cyclic. Since
$\ker\beta \cong \R^n$, we can deduce that $\alpha(G)$ is abelian-by-cyclic.

Evidently the intersection $\ker\alpha \cap L$ is trivial, which implies that $\ker\alpha$ acts freely on the Bass-Serre tree for $G$
expressed as an HNN-extension of $L$. Hence this kernel is free, so $G$ is free-by-abelian-by-cyclic.
\end{proof}

See~\cite{krop} for a stronger result in the case $n=1$.

\begin{thm} \label{thm:iscat0}
The group $G(A,L')=G(\Z^n,A,L')$ is a \cat{} group if and only if the
matrix $A$ is conjugate in $GL(n,\R)$ to an orthogonal matrix.
\end{thm}

\begin{proof}
Let $L=\Z^n$ and consider any $P\in GL(n,\R)$. Let $\Lambda\leqslant \R^n$ be the
lattice $P\,L$, let $\Lambda'\coloneq P\,L'$, and let $B\coloneq PAP^{-1}$.
There is a group isomorphism from $G(A,L')$ to $H\coloneq G(\Lambda,B\times,\Lambda')$, given by $c\mapsto P\,c$ for $c\in L$, and $t\mapsto t$.

Now suppose that $B$ is an orthogonal matrix.  In this case, there is
a homomorphism from $H\coloneq G(\Lambda,B\times,\Lambda')$ to the group of
isometries of $\E^n$, in which elements of $\Lambda$ act naturally as
translations and $t$ acts as multiplication by $B$.  This action is
not properly discontinuous, but its restriction to each conjugate of
$\Lambda$ is free and properly discontinuous.  Now let $T$ be the
Bass-Serre tree for $H$ expressed as an HNN-extension of $\Lambda$.
The stabilizer of each vertex of $T$ is a conjugate of $\Lambda$ and
the stabilizer of each edge of $T$ is a conjugate of $\Lambda'$.
Consider the diagonal action of $H$ on the product $\E^n\times T$.
Since edge and vertex stabilizers for the action of $H$ on $T$ act
freely properly discontinuously and cocompactly on $\E^n$, it
follows that the diagonal action of $H$ on the product $\E^n\times T$
is free, properly discontinuous, cocompact and isometric (for the
product metric on $\E^n\times T$, which is
\cat{}~\cite[Example~II.1.15(3)]{brihae}).
Hence $H$ is a \cat{} group.

For the converse, if $G=G(A,L')$ is a \cat{} group, then since $t$
is in $\comm_G(L)$, it follows that the action of $t$ preserves an
inner product on $L\otimes \R=\R^n$ by~Theorem~\ref{thm:addendum}.  But
this action is just multiplication by $A$.  Hence $A$ preserves an
inner product on $\R^n$ and so (since all $n$-dimensional
real inner product spaces are isomorphic), $A$ is conjugate in
$GL(n,\R)$ to an orthogonal matrix.
\end{proof}

\begin{rem}
There is another way to describe the \cat{} space constructed in the
above proof.  Suppose that $B$ is an orthogonal matrix, $\Lambda\leqslant \R^n$ is a
lattice and multiplication by $B$ induces an isomorphism of
finite-index sublattices $B\times:\Lambda'\rightarrow \Lambda''$. Then
multiplication by $B$ induces an isometry of tori from
$\T'\coloneq\R^n/\Lambda'$ to $\T''\coloneq\R^n/\Lambda''$.  Take the torus
$\T\coloneq\R^n/\Lambda$, and the direct product $\T'\times [0,1]\coloneq
\R^n/\Lambda' \times [0,1]$.  Glue the subspace $\T'\times\{0\}\cong
\T'$ to $\T$ via the covering map $\T'\rightarrow \T$, which is a local
isometry, and glue the subspace $\T'\times\{1\}\cong \T'$ to $\T$ by the
composite of multiplication by $B$ (an isometry $\T'\rightarrow \T''$)
and the covering map $\T''\rightarrow \T$.  By the gluing
lemma~\cite[II.11.13]{brihae}, the resulting space is locally \cat{}.
The universal cover of this space with its group of deck transformations
is of course the direct product $\E^n\times T$ with the isometric action
of $G(\Lambda,B\times,\Lambda')$ as described in the proof above.
\end{rem}

\begin{cor} \label{cor:quasi-isom}
If $A \in GL(n,\Q)$ is conjugate in $GL(n,\R)$ to an orthogonal matrix and $L' \neq L$  then $G(A,L')$ is quasi-isometric
to $\Z^n\times F$, the direct product of a free abelian group and a finite
rank non-abelian free group $F$.
\end{cor}

\begin{proof} Let $F$ be the free group of rank $m=|L:L'|>1$.
Since the determinant of $A$ is $\pm 1$, $m=|L:L''|$ too.
Hence the Bass-Serre tree $T$, for the decomposition of $G(A,L')$ as an HNN-extension, is a regular tree of valency
$2m$, and so both $G$ and $\Z^n\times F$ have natural isometric geometric actions
on $\E^n\times T$.  By the \v{S}varc-Milnor lemma \cite[I.8.19]{brihae}
they are quasi-isometric to each other.
\end{proof}

\begin{thm}\label{thm:irredlat}
Suppose that $G=G(A,L')$ where $A$ has infinite order and is conjugate
in $GL(n,\R)$ to an orthogonal matrix.  Then $G$ is a lattice in
$\isom(\E^n)\times\isom(T)$ whose projections to the factors are
not discrete.
\end{thm}

\begin{proof}
Since $A$ is conjugate to a matrix in $O(n)$, $G$ acts isometrically
on $\E^n$, and the action on the Bass-Serre tree $T$ is always isometric.
Since $G$ acts freely, properly discontinuously, cocompactly and
isometrically on
$\E^n\times T$, it follows that $G$ is a lattice in $\isom(\E^n)\times\isom(T)$.  Since $A$ has infinite
order, the element $t\in G$ acts on $\E^n$ as an infinite order element of
the point stabilizer, which is compact (and isomorphic to $O(n)$), showing
that the projection of $G$ to $\isom(\E^n)$ is not discrete.

The vertex stabilizers of the action of $G$ on $T$ are conjugates of the
subgroup $L\leqslant G$, hence the kernel
of this action is the core $c(L)$ of $L$, i.e., the intersection of
all conjugates of $L$ in $G$.  Provided that $c(L)$ has infinite index
in $L$, it will follow that the image of $L$ in $\isom(T)$ is an
infinite subgroup of the vertex stabilizer, which is a compact (profinite)
group (because the tree $T$ is locally finite).
But $c(L)$ is a normal abelian subgroup of $G$, and so,
by the Flat Torus Theorem (Theorem \ref{thm:flattorus}),
there is $k \in \N$ such that $t^k$ centralizes $c(L)$, and
thus $t^k$ acts trivially on $c(L)\otimes \R$.  But $t^k$ cannot
act trivially on $L\otimes \R$ because $A$ has infinite order, and
so $c(L)\otimes \R$ must be a proper subspace of $L\otimes \R$. It follows that
$|L:c(L)|=\infty$, and so the image of $L$ (and, hence, of $G$)
in $\isom(T)$ is not discrete.
\end{proof}

In the case when $n=2$, it follows that $G$ as in the above statement
is an irreducible lattice in $\isom(\E^2)\times\isom(T)$, because
the matrix $A$ acts irreducibly on $\E^2$.  This is
not necessarily the case for larger $n$.  For example, if $G$ satisfies
the hypotheses, then so does
$\Z\times G\leqslant \isom(\E^1)\times\isom(\E^n)\times \isom(T)
\leqslant\isom(\E^{n+1})\times \isom(T)$.  As mentioned in the introduction,
the existence of irreducible lattices in $\isom(\E^2)\times\isom(T)$
contradicts~\cite[Theorem~1.3(i), Proposition~3.6, Theorem~3.8]{capmon}.

\section{Characterizing biautomaticity of the groups \texorpdfstring{$G(A,L')$}{ G(A,L')}}\label{sec:biaut}
Suppose that  $L=\Z^n$, for some $n \in \N$, $A \in GL(n,\Q)$ and $L' \leqslant L$ is a finite-index subgroup such that $L'' \coloneq A\,L'$ is contained in $L$. In this section we study the (virtual) biautomaticity of the groups $G(A,L')$ defined in \eqref{eq:pres_of_G(A,L')}.

\begin{lemma} \label{lem:selfcent}
If $L'$ and $L''$ are both proper subgroups of $L$ then
$L$ is self-centralizing in $G(A,L')$.  In particular, if
neither $A$ nor $A^{-1}$ is an integer matrix, then $L$ is
self-centralizing.
\end{lemma}

\begin{proof}
Let $T$ be the Bass-Serre tree for $G=G(A,L')$ expressed as an
HNN-extension of $L$.  The centralizer of $L$ will act on the set of
$L$-fixed points $T^L$, which is a subtree (because the unique
geodesic path between two fixed points must also be fixed).
The vertex corresponding to the identity coset of $L$ is
fixed by $L$, but the hypotheses imply that no edge that is
incident with this vertex can be fixed by $L$.  Hence
the fixed point set for the action of $L$ on $T$ is this
single vertex.  The first claim
follows since $L$ is the full stabilizer of this vertex.

The second claim follows from the first, because the largest
possible choices for $L'$ and $L''$ are $L'= A^{-1}\,L\cap L$
and $L''= A\,L\cap L$.
\end{proof}

Recall that a subgroup $G$ of a direct product $H \times F$, of two groups $H$ and $F$, is said to be a \emph{subdirect product} if the restrictions to $G$ of the natural projections  $H \times F \to H$ and $H \times F \to F$ are surjective.

The following criterion for biautomaticity will be useful.
\begin{prop}\label{prop:crit_biaut} Let $H$ be a finitely generated virtually abelian group and let $F$ be a biautomatic group. Then every subdirect product $G \leqslant H \times F$ is also biautomatic.
\end{prop}

\begin{proof} Abusing the notation we identify any subgroup $S \leqslant H$ with the subgroup $S \times \{1\}$ of $H \times F$.
Since $G \leqslant H \times F$ is subdirect, the subgroup $N\coloneq G \cap H$ is normal in $H$ (cf. \cite[Lemma~2.1]{Min-fibre}). Now, by \cite[Lemma~4.2]{Min-VR}, there exists a normal subgroup
$R \lhd H$ which intersects $N$ trivially and such that $|H:NR|<\infty$. Let $H_1\coloneq H/R$ and $\phi: H \times F \to H_1 \times F$ be the natural homomorphism whose kernel is $R$.

By construction, $R$ has trivial intersection with $G$ in $H\times F$, hence $\phi(G) \cong G$. Evidently $\phi(G)$ is still subdirect in $H_1 \times F$.
Moreover, $\phi(N) \subseteq \phi(G) \cap H_1$ has finite index in $H_1$ as $|H:NR|<\infty$, which implies that $\phi(G)$ has finite index in $H_1\times F$ (see  \cite[Lemma~2.1]{Min-fibre}).

Now, $H_1$ is finitely generated and virtually abelian, so it is biautomatic by \cite[Corollary~4.2.4]{Eps}. Therefore $H_1 \times F$ is biautomatic by \cite[Theorem~4.1.1]{Eps}, and, hence its finite-index subgroup $\phi(G) \cong G$ is biautomatic by \cite[Theorem~4.1.4]{Eps}.
\end{proof}

\begin{thm} \label{thm:notbiaut}
The group $G=G(A,L')$ is biautomatic if and only if $A$ has
finite order.
\end{thm}

\begin{proof} Assume that $G$ is biautomatic.
If either $L'=L$ or $L''=L$, then $G$ is an ascending HNN-extension
of $L$; in this case Groves and Hermiller \cite[Main~Theorem]{grhe} proved that $G$ must be virtually abelian. The latter clearly implies that $A$ has finite order in $GL(n,\Q)$.

Thus we can assume $L'$ and $L''$ are proper subgroups of $L$.
Then $L$ is finitely generated and self-centralizing by Lemma~\ref{lem:selfcent}, and the
commensurator $\comm_G(L)$ is the whole of $G$. Therefore, by Corollary~\ref{cor:comm_of_centralizer}, there is $k \in \N$ such that
$t^k$ centralizes a finite-index subgroup of $L$. This means that $A^k$ is the identity matrix, and so $A$ has order
dividing $k$.

Now suppose that $A$ has finite order $k \in \N$, and let $M$
be the intersection of the subgroups $t^i L' t^{-i}=A^i\,L'\leqslant L$, $i=0,\dots,k-1$.  Then
$M$ is a finite-index subgroup of $L$ and is normal
in $G$.  The quotient group $F=G/M$ is an
HNN-extension of the  finite group $L/M$, hence it is virtually free. It follows that $F$ is word hyperbolic, and, therefore, biautomatic (this can be easily deduced from \cite[Chapter 3]{Eps}; see \cite[III.$\Gamma$.2.20]{brihae} for an explicit statement).
Let $\beta: G \to F$ denote the natural epimorphism with $\ker\beta=M$.

As in the proof of Proposition~\ref{prop:free-by-ab-by-cyc}, we also have a homomorphism $\alpha: G \to AGL(n,\R)$ which sends $L$ to a subgroup of translations of $\R^n$ and $t$ to the linear transformation of $\R^n$ corresponding to $A$. Since $A$ has finite order, it is clear that $H \coloneq \alpha(G)$ is virtually abelian; moreover, $L \cap \ker\alpha=\{1\}$ by construction.

Define the homomorphism $\psi: G \to H \times F$ by $\psi(g)=(\alpha(g),\beta(g))$ for all $g \in G$. This homomorphism is injective because the kernels of $\alpha$ and $\beta$ intersect trivially. Since $\alpha(G)=H$ and $\beta(G)=F$,
$\psi(G)$ is a subdirect product in $H \times F$. Therefore $G \cong \psi(G)$ is biautomatic by Proposition~\ref{prop:crit_biaut}.
\end{proof}

A well-known open problem
(see \cite[Open Question~4.1.5]{Eps}) asks whether a group which has a finite-index biautomatic subgroup must itself be biautomatic. In the remainder of this section we will show that this is indeed the case for our groups:  $G(A,L')$ is biautomatic if and only if it is virtually biautomatic.

\begin{lemma} \label{lem:morecents-new} Suppose that $L'$ and $L''$ are both proper subgroups of $L$ and let $G^{(2)}$ denote the
second derived subgroup of $G=G(A,L')$. Then $G^{(2)}$ is a non-abelian free group and for any two non-commuting elements $g_1,g_2 \in G^{(2)}$, the centralizer $\mathrm{C}_G(\{g_1,g_2\})$ is a finite-index subgroup of $fLf^{-1}$, for some $f \in G$.
\end{lemma}

\begin{proof} Proposition \ref{prop:free-by-ab-by-cyc} implies that $G^{(2)}$ is free, and since $L'$ and $L''$ are proper subgroups of $L$, $G$ cannot be soluble
(it will contain non-abelian free subgroups being an HNN-extension in which both of the associated subgroups are proper subgroups of the base group),
hence $G^{(2)}$ is non-abelian.

Observe that the normal closure $N$, of $L$ in $G$, is generated by the elements $t^i c t^{-i}$, where $i \in \Z$ and $c \in L$. Evidently any such element centralizes the finite-index subgroup $t^i L t^{-i} \cap L$, of $L$. Since each $g \in N$ is a product of finitely many such elements, we conclude that $g$ must also centralize
a finite-index subgroup of $L$ in $G$.

Consider any two non-commuting elements $g_1,g_2 \in G^{(2)}$. Since $G/N$ is cyclic (generated by the image of $t$), $G^{(2)} \subseteq N$, so $g_1$ and
$g_2$ both centralize some finite-index subgroup $K$ of $L$. Let $T$ be the Bass-Serre tree for the splitting of $G$ as an HNN-extension of $L$.
Note that the subgroup $\langle g_1,g_2\rangle \subseteq G^{(2)}$ acts
freely on $T$
(see the proof of Proposition~\ref{prop:free-by-ab-by-cyc}), so each $g_j$ acts
as a hyperbolic isometry of $T$ with an axis $\ell_j$, $j=1,2$.

If $\ell_1=\ell_2$ then the rank $2$ free subgroup $\langle g_1,g_2\rangle$
acts on the simplicial line $\ell_1$ by isometries. This action must have a
non-trivial kernel, because the group of all simplicial isometries of $\ell_1$
is isomorphic to the infinite dihedral group. This means that a non-trivial
element of $\langle g_1,g_2\rangle$
fixes $\ell_1$ pointwise, contradicting  the freeness of the action
of this subgroup on $T$.

Hence $\ell_1$ and $\ell_2$ must be distinct. Since $\mathrm{C}_G(\{g_1,g_2\})$
preserves each of these axes setwise, this centralizer must fix a vertex of $T$:
if $\ell_1 \cap \ell_2$ is a finite segment, it will fix a vertex of this
segment; if $\ell_1 \cap \ell_2$ is an infinite ray, it will fix all of it;
finally, if
$\ell_1 \cap \ell_2=\emptyset$, it will fix all vertices of the unique
geodesic segment connecting these two axes.
The vertex stabilizers for the action of $G$ on $T$ are conjugates of $L$,
so there exists $f \in G$ such that $\mathrm{C}_G(\{g_1,g_2\}) \subseteq fLf^{-1}$. Recall that $G$ commensurates $L$,
hence $ L \cap fLf^{-1}$ has finite index in $fLf^{-1}$.
Since $\mathrm{C}_G(\{g_1,g_2\})$ contains $K$ and $|L:K|<\infty$, we conclude that $|fLf^{-1}:\mathrm{C}_G(\{g_1,g_2\})|<\infty$, as claimed.
\end{proof}

\begin{thm}\label{thm:notvirtbiaut}
If $A$ has infinite order then the group $G=G(A,L')$ is not virtually biautomatic.
\end{thm}

\begin{proof}
As in the proof of Theorem~\ref{thm:notbiaut}, the case when either $L=L'$ or
$L=L''$ follows from the result of Groves-Hermiller~\cite[Main Theorem]{grhe}, so we assume
from now on that both $L'$ and $L''$ are proper subgroups of $L$.

Let $H$ be a finite-index subgroup of $G$, then $|G^{(2)}:G^{(2)}\cap H|<\infty$, so $G^{(2)}\cap H$  is a non-abelian free subgroup by Lemma~\ref{lem:morecents-new}.
Choose arbitrary two non-commuting elements $g_1,g_2 \in H \cap G^{(2)}$. The same lemma states that $\mathrm{C}_G(\{g_1,g_2\})$ is a
finite-index subgroup of $fLf^{-1}$ for some $f \in G$. Since $G$ commensurates $L$, it also commensurates $fLf^{-1}$, as well as its finite-index subgroup
$M\coloneq \mathrm{C}_H(\{g_1,g_2\})=H \cap \mathrm{C}_G(\{g_1,g_2\})$. It follows that $M$ is an abelian subgroup commensurated by $H$.
Now, Corollary \ref{cor:comm_of_centralizer} implies that $H$ cannot be biautomatic as $ft^lf^{-1} \in H$, for some $l \in \N$, and no non-trivial power of this element can centralize a finite-index subgroup of $fLf^{-1}$ (since $A^l$ has infinite order).
\end{proof}

\section{Explicit examples} \label{sec:2dim}

Throughout this section, it will be sufficient to specialize the
groups $G(A,L')$, defined in Section~\ref{sec:comm_HNN}, to the case when $L=\Z^2$ has rank two. We will write
$M_2(\Q)$ to denote the ring of $2\times 2$ matrices with
rational entries.  Before starting, we recall that the classification,
up to conjugacy, of square matrices over a field $k$ is equivalent to
the classification, up to isomorphism, of finitely generated torsion
modules for the polynomial ring $k[x]$, which is a principal ideal
domain~\cite[Ch.~XI]{macbir}.  In particular, if $f(x)\in k[x]$ is
a polynomial that is square-free (i.e., not divisible by the square
of any irreducible polynomial) then there is exactly one conjugacy
class of square matrices over $k$ with characteristic polynomial
$f(x)$: this is the analogue for $k[x]$ of the familiar statement
(for $\Z$-modules) that there is exactly one abelian group of order $n$
provided that $n$ is square-free.

\begin{prop}\label{prop:crit_for_2x2_matr}
If $A\in M_2(\Q)$, then $A$ is conjugate to an element of $SO(2)$ in $GL(2,\R)$
if and only if $\det(A)=1$ and either $A=\pm I$ or
 $-2<\tr(A)<2$.  Such a matrix $A$ has finite order if and only if
$\tr(A)\in\Z$.
\end{prop}

\begin{proof}
Matrices in $SO(2)$ have the claimed properties, and these are
not changed by conjugation.  Conversely, if $A$ has the claimed
properties and $A\neq \pm I$, then the characteristic polynomial
of $A$ has the form $X^2-\tr(A)X +1$, and is irreducible over $\R$.
Any two matrices with this characteristic polynomial are conjugate
in $GL(2,\R)$.

If $A$ has finite order, the additive group of the subring of
$M_2(\Q)$ generated by $A$ is finitely generated, from which it
follows that the characteristic polynomial of $A$ must lie in $\Z[X]$.
For the converse, the choices of $-2,-1,0,1,2$ for $\tr(A)$ give
rise to elements of order $2,3,4,6,1$ respectively.
\end{proof}

\begin{ex}\label{ex:mats}
As examples, the matrix
\begin{equation}\label{eq:A_k/m}
A_{k/m}\coloneq\begin{pmatrix} 0&-1 \\ 1& k/m \\ \end{pmatrix},
\end{equation}
for $k,m\in \Z$ with $m>0$, is conjugate to a matrix in $SO(2)$
if and only if $-2m<k<2m$, and has infinite order provided
that $k\neq -m,0,m$.

Pythagorean triples give rise to matrices of infinite order
in $GL(2,\Q)\cap SO(2)$.
For example we shall consider the matrix $A_P$, defined by
\begin{equation}\label{eq:A_P}
A_P\coloneq \begin{pmatrix} 3/5& -4/5 \\ 4/5 & 3/5 \\ \end{pmatrix}.
\end{equation}

\end{ex}

A combination of Theorem~\ref{thm:iscat0} with Theorem \ref{thm:notbiaut} yields the following.
\begin{cor}
If $A\in GL(2,\Q)$ has infinite order and is conjugate to an element of $SO(2)$  in $GL(2,\R)$
then for any $L'$, $G(A,L')$ is \cat{} and is not biautomatic.
\end{cor}

\begin{ex}\label{ex:groups}
For more concrete examples, consider the groups
\[G_{k,m}\coloneq\langle a,b,t \,\|\, [a,b]=1,\,\, tat^{-1}=b,\,\,
tb^m t^{-1}= a^{-m}b^k \rangle.\]

This group is \cat{} whenever $-2m<k<2m$ and is not biautomatic
provided that $k\neq -m,0,m$.  The group $G_{k,m}$ is of the form
$G(A_{k/m},L')$, where $L'=\langle (1,0)^T, (0,m)^T\rangle$ has index $m$ in $L$.  In the case when ${\rm gcd}(k,m)=1$,
$L'$ is as large as possible, so $G_{k,m}=G(A_{k/m})$.

The first example of a group of this type that we found was the
group $G'_P\coloneq G(A_P,(5\Z)^2)$, where $A_P$ is the matrix
defined in \eqref{eq:A_P} .
Here are presentations for the groups $G_P\coloneq G(A_P)$
mentioned in the introduction (which corresponds to the case
$L'=L \cap A_P^{-1}\,L=\langle (2,-1)^T,(1,2)^T\rangle$) and $G'_P$.

\[G_P=\langle a,b,t \,\|\, [a,b]=1,\,\, ta^2b^{-1}t^{-1}=a^2b,\,\,
tab^2t^{-1}= a^{-1}b^2 \rangle,\]
\[G'_P=\langle a,b,t \,\|\, [a,b]=1,\,\, ta^5t^{-1}=a^3b^4,\,\,
tb^5t^{-1}= a^{-4}b^3 \rangle
.\]

\end{ex}

In the case when $G(A,L')$ is \cat{}, there is usually only one choice
of \cat{} metric on $L\otimes \R$ up to homothety.

\begin{cor}
  Suppose that $A\in GL(2,\Q)$ has order at least 3 and is conjugate to an
  element of $SO(2)$   in $GL(2,\R)$.  In this case, the inner product
  $\langle \,\cdot\,,\,\cdot\,\rangle_{L}$ on $L\otimes \R$, defined by
\eqref{eq:inner_prod_formula} when viewing $L$ as a subgroup of the
\cat{} group $G(A,L')$, is, up to multiplication by a scalar, the unique
inner product that is preserved by $A$.
\end{cor}

\begin{proof}
  Let $a\in L$ be any non-identity element, and suppose that the
  \cat{} structure on $G(A,L')$ is chosen so that   $a$ acts on
  $L\otimes \R$ as translation by some distance $\lambda>0$.  In
  this case $\langle a,a\rangle_L=\lambda^2$, but also
  \[ \langle A\,a,A\,a\rangle_L=\lambda^2, \quad
  \langle A\,a,a\rangle_L = \langle a,A\,a\rangle_L = \lambda^2 \tr(A)/2,\]
  because $A$ acts on $L\otimes \R$ as rotation through an angle
  $\theta$ with $2\cos(\theta)=\tr(A)$.  The uniqueness follows, because
  for $A\neq \pm I_2$, $a$ and $A\,a$ form a basis of $L\otimes \R$.
\end{proof}

Figure~\ref{fig:fig2} below depicts the unique geometries on $L\otimes \R$
for the seven \cat{} groups $G_{k,2}$ and the \cat{} group $G_P$.

\begin{figure}[htp]
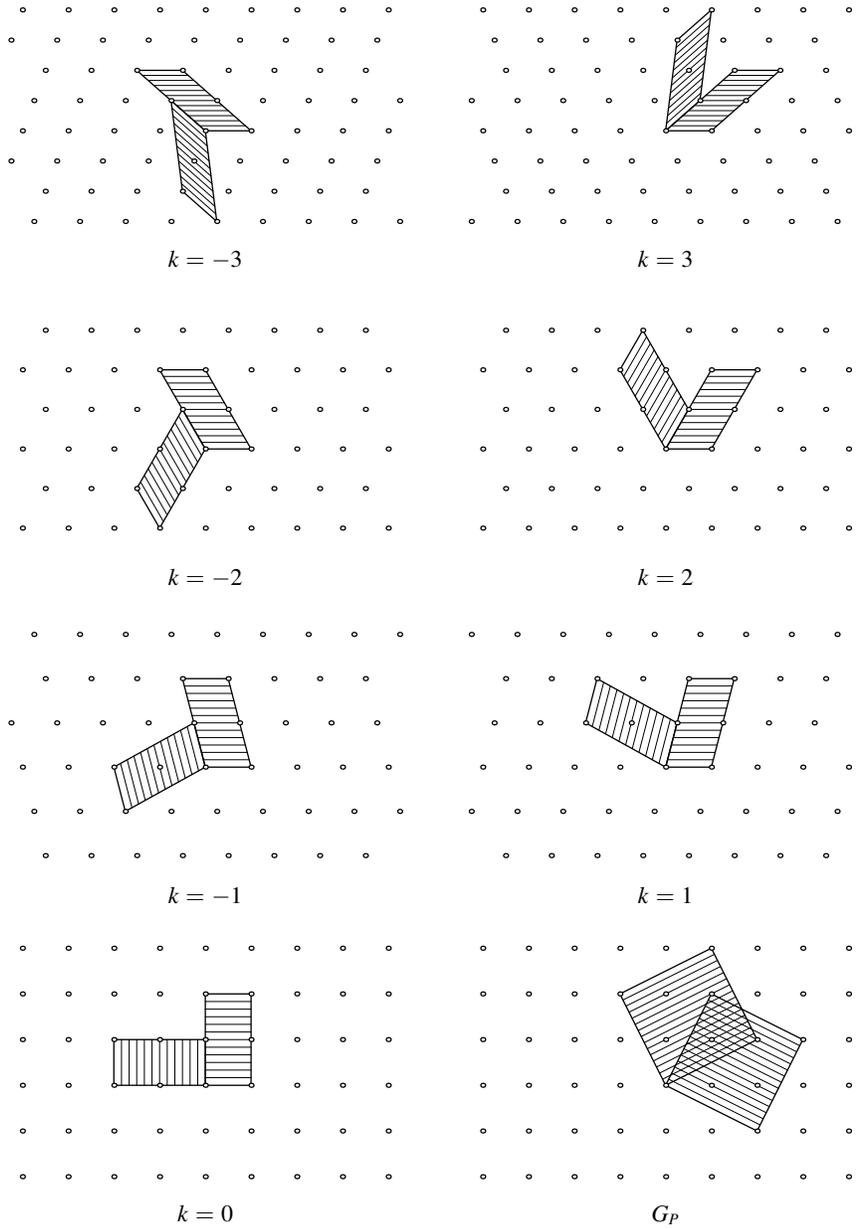

\captionsetup[subfigure]{labelformat=empty}
\centering
%%PLEASE CHANGE THE HEIGHT OF THESE GRAPHICS IF NECESSARY TO ENSURE
%%THAT THE WHOLE OF FIGURE 3 WITH CAPTION JUST FITS ON ONE PAGE.
\subfloat[$k=-3$]{\includegraphics[height=33mm]{pics/r2-3.mps}}\qquad
\subfloat[$k=3$]{\includegraphics[height=33mm]{pics/r23.mps}}

\subfloat[$k=-2$]{\includegraphics[height=33mm]{pics/r2-2.mps}}\qquad
\subfloat[$k=2$]{\includegraphics[height=33mm]{pics/r22.mps}}

\subfloat[$k=-1$]{\includegraphics[height=33mm]{pics/r2-1.mps}}\qquad
\subfloat[$k=1$]{\includegraphics[height=33mm]{pics/r21.mps}}

\subfloat[$k=0$]{\includegraphics[height=33mm]{pics/r20.mps}}\qquad
\subfloat[$G_P$]{\includegraphics[height=33mm]{pics/rP.mps}}

\caption{The geometry of $L\otimes \R$ in the \cat{} groups $G_{k,2}$ and
  $G_P$.  Dots represent points of $L$ and the shaded regions
  represent the fundamental domains for $L'$ and $L''$ that are
  implied by the given presentations.
  In each picture $a$ acts as a horizontal translation and $t$
  acts as rotation through $\arccos(k/4)$ for~$G_{k,2}$ and
  $\arccos(3/5)$ for $G_P$.} \label{fig:fig2}

\end{figure}

In \cite[Question~2.7]{Bestvina-list} D. Wise asked whether every \cat{} group $G$ has
the following property: for any elements $a,b\in G$, there exists
$n>0$ so that the subgroup $\langle a^n,b^n\rangle$ is either
abelian or free.

\begin{cor}\label{cor:wise}
If $A$ has infinite order and is conjugate to an element of $SO(2)$
in $GL(2,\R)$ then the group $G=G(A,L')$, for any suitable
choice of $L'$, is \cat{} but it is not virtually biautomatic and it does not have Wise's property. In particular, this applies to the groups $G_P$, $G'_P$ and $G_{k,m}$ from Example~\ref{ex:groups}, provided $-2m<k<2m$ and $k \neq 0,\pm m$.
\end{cor}

\begin{proof} The group $G$ is \cat{} by Theorem \ref{thm:iscat0} and it is not virtually biautomatic by Theorem~\ref{thm:notvirtbiaut}.

Let $a\in L$ be a non-identity element of $L\leqslant G$ and let $t \in G$ be
the stable letter from the presentation \eqref{eq:pres_of_G(A,L')}. Set  $m \coloneq |L/L'|$, then $b^m \in L'$ for every $b \in L$.
Given any $n \in \N$, the subgroup
$H_n=\langle a^n,t^n\rangle \leqslant G$ cannot be abelian because $t^n$ does not centralize
any non-identity element of $L$.  On the other hand, $H_n$ cannot be
free, because it contains the element $t^n a^{nm^n} t^{-n}\in L$,
which together with $a^n$ generates a finite-index subgroup of $L$.
\end{proof}

\begin{rem} Although the standard Tits alternative is still unknown for general \cat{} groups,
Proposition~\ref{prop:free-by-ab-by-cyc} implies that it does hold for any of the groups $G(A,L')$, defined by \eqref{eq:pres_of_G(A,L')}.
\end{rem}
%
%Indeed, let $H \leqslant G$ be any subgroup. By definition, $G$ admits an action on a simplicial tree $T$ with abelian vertex stabilizers, and the standard classification of group actions on trees (see, for example, \cite[Theorem~2.7]{C-M}) implies that if  $H$ contains no non-abelian free subgroups then
%the induced action of $H$ on $T$ is reducible. The latter means that either $H$ fixes a vertex of $T$ or $H$ fixes an end of $T$, or $H$ leaves some line in $T$ invariant.
%Since vertex stabilizers for the action of $H$ on $T$ are abelian, it follows that $H$ is either abelian, or  abelian-by-cyclic, or abelian-by-dihedral.

\begin{rem} After hearing the first named author's talks on the results of this paper, M. Bridson suggested that the methods developed in his paper \cite{Bridson-combings} give an alternative proof that the groups $G(A,L')$ from Corollary~\ref{cor:wise} are not biautomatic. Indeed, \cite[Proposition~2.2]{Bridson-combings} states that any biautomatic structure on an abelian group can contain only finitely many commensurability classes of quasiconvex subgroups. If the matrix $A$ has infinite order and has no rational eigenvectors, then for any infinite cyclic quasiconvex subgroup $C\leqslant L$, its conjugates $t^i Ct^{-i}$, $i \in \N$, will all be quasiconvex, pairwise non-commensurable and will virtually be subgroups of $L$.
Moreover, using the work of Bridson and Gilman \cite{bri-gil} it may be possible to extend this method to prove the stronger statement that $G(A,L')$ does not admit any bounded bicombing such that the corresponding language is context-free.
\end{rem}

\section{Residual finiteness and non-Hopficity}\label{sec:n-H}
As mentioned in the introduction, the groups $G(A,L')$ are higher
dimensional generalizations of the Baumslag-Solitar groups.
Originally Baumslag and Solitar introduced their groups in \cite{Baum-Sol} as
the first examples of  non-Hopfian  one-relator groups.
It is not hard to see that nearly the same argument shows that many of our
groups are non-Hopfian.  The following result is closely related to
a theorem of D. Meier~\cite{meier}, although neither result is a direct
corollary of the other.

\begin{prop} \label{prop:nonhopf}
Let $A\in GL(2,\Q)$ satisfy $\det A \in \Z$.
Suppose that  there exists an integer
$m>1$ so that $mA$ is an integer matrix and $k\coloneq m \tr(A)$ is coprime to $m$.  Then the group $G=G(A)$, defined in Section~\ref{sec:comm_HNN}, is
non-Hopfian.
\end{prop}

\begin{proof}
The characteristic polynomial
of $A$ is $X^2-(k/m)X+l$, where $l\coloneq \det A \in \Z$. This implies that
\begin{equation}\label{eq:comb_A_and_A-1}
mA+ml A^{-1}=k I_2,
\end{equation}
in particular, $ml A^{-1}$ is an integer matrix.
Since $mA$ is an integer matrix, we have
$A\,L^m=(mA)\,L \subseteq L$, hence $L^m\subseteq  A^{-1}\,L\cap L=L'$. Similarly, $L^{ml}\subseteq  A\,L\cap L=L''$. Thus
\begin{equation}\label{eq:L^m_in_L'_and_L''}
L^m\subseteq L' ~\mbox{ and }~ L^{ml} \subseteq L''.
\end{equation}

Let $a,b\in L\leqslant G$ be generators for $L \cong \Z^2$, and let $t\in G$ be
the stable letter.   Combining \eqref{eq:L^m_in_L'_and_L''} with \eqref{eq:comb_A_and_A-1}, we obtain the identity
\begin{equation}\label{eq:identity}
tc^mt^{-2}c^{ml}t= c^k~ \mbox{ for any } c\in L \mbox{ in } G.
\end{equation}

It is easy to check that the map $\phi$ defined by
\[\phi(a)\coloneq a^m,\quad \phi(b)\coloneq b^m,\quad \phi(t)\coloneq t\]
extends to an endomorphism $\phi:G\rightarrow G$. Since the image of
$\phi$ contains $t$ and $L^m=\langle a^m, b^m \rangle$, in view of \eqref{eq:identity} we see that it also contains
$a^k$ and $b^k$. Recall that $k$ is coprime to $m$ by the assumptions, hence the image of $\phi$ contains $a$ and $b$, so $\phi$ is
surjective.

It remains to show that $\phi$ is not injective.
The assumptions imply that both $A$ and $A^{-1}$ have some non-integer entries, hence $L'$ and $L''$ must be proper subgroups of $L$.
Choose arbitrary $c\in L-L'$ and $d\in L-L''$ (since $L$ cannot be the union of the
two proper subgroups $L'$ and $L''$, there are  elements of this form with $c=d$). Then the commutator
$[tct^{-1},d]$ is non-trivial in $G$ by Britton's Lemma for HNN-extensions (cf. \cite[Section IV.2]{L-S}), but it
is in the kernel of $\phi$ by \eqref{eq:L^m_in_L'_and_L''}.
\end{proof}

By a well-known theorem of Malcev (cf. \cite{Mal}) the groups from Proposition~\ref{prop:nonhopf} cannot be residually finite.
We can actually say more about the finite images of such groups.

\begin{cor}\label{cor:G-by-R-metab} Suppose that $G=G(A)$ is a group
satisfying the assumptions of Proposition~\ref{prop:nonhopf},
and $\phi:G \to G$  is the endomorphism defined in the proof of this proposition.
Let $R\coloneq \bigcup_{n=1}^\infty \ker(\phi^n) \lhd G$. Then the quotient
$G/R$ is abelian-by-cyclic. In particular, every finite quotient of $G$ is metabelian.
\end{cor}

\begin{proof}
Let $N$ be the normal closure of $L$ in $G$, so that $G/N$ is
infinite cyclic.
Any $g, h \in N$ can be written as products of elements of the form
$t^i c t^{-i}$, where $c \in L$ and  $i \in \Z$. Hence, we can choose $s \in \N$ such that $g'\coloneq t^s g t^{-s}$ and $h'\coloneq t^s h t^{-s}$ are products of elements of the form
$t^i c t^{-i}$, where $c \in L$ and  $i>0$.

Note that, in view of \eqref{eq:L^m_in_L'_and_L''}, if
$i \in \N$ then for each $j \ge i$,
$\phi^j(t^i c t^{-i})\in L$ in $G$. Therefore there exists a sufficiently large
$n \in \N$ such that $\phi^n(g') \in L$ and $\phi^n(h') \in L$.
Since $L$ is abelian, it follows that the commutator
$[g',h'] \in \ker\phi^n\subseteq R$, so $[g,h]=t^{-s}[g',h']t^s \in R$, for arbitrary $g,h \in N$. Therefore
the image of $N$ in $G/R$ is abelian, so $G/R$ is abelian-by-cyclic.

For the last assertion, recall that the proof of Malcev's theorem implies that $R$ is contained in the intersection of all finite-index subgroups of $G$.  Therefore it is annihilated
by every epimorphism $\psi:G\to Q$, with $Q$ finite. It follows that $Q$ is a quotient of $G/R$, so it is also abelian-by-cyclic, as claimed.
\end{proof}

\begin{cor}
Suppose that $G=G_P$ or $G=G_{k,m}$, with $m >1$, $-2m<k<2m$ and ${\rm gcd}(k,m)=1$,
is a group  from Example~\ref{ex:groups}.
Then $G$ is a \cat{} group which is not Hopfian and not uniformly non-amenable.
\end{cor}

\begin{proof}  The group $G$ is non-Hopfian by Proposition~\ref{prop:nonhopf}.
The fact that $G$ is not uniformly non-amenable follows from Corollary~\ref{cor:G-by-R-metab} by \cite[Corollary~13.2]{ABLRSV} or \cite[Theorem~2.2]{Osin-Kazhdan}.
\end{proof}

The fact that the group $G_P$ is non-Hopfian can also be derived from Meier's criterion
\cite[Lemma~1]{meier}, however this criterion does not seem to apply to the groups $G_{k,m}$.  The first examples of non-Hopfian \cat{} groups were constructed by Wise in \cite{Wise-non-Hopf}.

Using the work of Andreadakis, Raptis and Varsos \cite{A-R-V} we can characterize the residual finiteness of groups $G(A,L')$ in general.

\begin{prop} \label{prop:rf-crit} Suppose that $L = \Z^n$, $A \in GL(n,\Q)$ and $L'$ is a finite-index subgroup of $L$ such that $A\,L' \subseteq L$.
Then the group $G= G(A,L')$, defined by \eqref{eq:pres_of_G(A,L')}, is residually finite if and only if  one of the following conditions holds:
\begin{itemize}
  \item[(i)] $L'=L$ or $A\,L'=L$;
  \item[(ii)] $A$ is conjugate in $GL(n,\Q)$ to a matrix from $GL(n,\Z)$.
\end{itemize}
\end{prop}

\begin{proof} By \cite[Theorem 1]{A-R-V} the group $G$, defined by \eqref{eq:pres_of_G(A,L')}, is residually finite
if and only if either $L=L'$ or $L=A\,L'$ (in which case $G$
is metabelian) or $t$ normalizes a finite-index subgroup of $L$. Let us prove that the latter is equivalent
to saying that $A$ is conjugate in $GL(n,\Q)$ to a matrix from $GL(n,\Z)$.

Suppose, first, that $tMt^{-1}=M$ for some finite-index subgroup $M \leqslant L$. This implies that $M \subseteq L'$, so
$tMt^{-1}=A\,M=M$. Evidently $M=B \,L$ for some invertible matrix $B$ with integer entries,
thus $(B^{-1}\,A\, B)\,L=L$, i.e., $B^{-1}\,A\, B \in GL(n,\Z)$.

Conversely, assume that $C^{-1}\,A\,C \in GL(n,\Z)$ for some $C \in GL(n,\Q)$. Set $k \coloneq |L/L'|$ and choose $m \in \N$ so that all entries of the matrix
$B\coloneq mC$ are integers divisible by $k$. Then $B$ is invertible, so $M \coloneq B \,L$ has finite index in $L$; moreover, $M \subseteq L'$ by the choice of $m$.
Note that $B^{-1}=\frac{1}{m} C^{-1}$, so $B^{-1}\,A\, B=C^{-1}\,A\,C \in GL(n,\Z)$. It follows that $(B^{-1}\,A\, B)\,L=L$, hence
$M=A\,M=tMt^{-1}$, as required.
\end{proof}

Using the rational canonical form for matrices~\cite[Chapter~XI.4]{macbir},
condition (ii) from Proposition~\ref{prop:rf-crit} can be restated more
algebraically.
\begin{rem}\label{rem:crit_for_conj_to_GL(n,Z)} A matrix $A \in GL(n,\Q)$
is conjugate in $GL(n,\Q)$ to some matrix from $GL(n,\Z)$
if and only if $\det(A)=\pm 1$ and
all coefficients of the characteristic polynomial of
$A$ are integers.
\end{rem}

\begin{prop}\label{prop:linearity} Let $G=G(A,L')$, where $L=\Z^n$, $A \in GL(n,\Q)$ and $L'$ is a finite-index subgroup of $L$ such that $A\,L' \subseteq L$. Then $G$ is residually finite if and only if $G$ is linear over $\Q$.
\end{prop}

\begin{proof}
Finitely generated linear groups are residually finite by a result of Malcev \cite{Mal}, hence we only need to prove that if $G$ is residually finite then it is isomorphic to a subgroup of $GL(m,\Q)$ for some $m \in \N$.

Suppose, first, that $t$ normalizes a finite-index subgroup $M$ of $L$. Then $M \lhd G$ and $F \coloneq G/M$ is an HNN-extension of the finite group $L/M$. Thus $F$ is finitely generated and virtually free, so it is linear over $\Q$.  Let $\beta:G \to F$ be the natural epimorphism with $\ker\beta=M$.

As before, we also have a homomorphism $\alpha:G \to AGL(n,\Q)$, which comes from the actions of $L$ and $t$ on $\Q^n$ by translations and  by $A$ respectively. The standard embedding of  $AGL(n,\Q)$ in $GL(n+1,\Q)$ shows that it is linear over $\Q$.

Evidently $\ker \alpha \cap L =\{1\}$, hence $\ker\alpha \cap \ker\beta=\{1\}$. Therefore the homomorphism $\psi:G \to  AGL_n(\Q) \times F$, defined by $\psi(g)=(\alpha(g),\beta(g))$ for all $g \in G$, is injective. It follows that $G$ is linear over $\Q$.

If $t$ does not normalize any finite-index subgroup of $L$, then, by  \cite[Theorem 1]{A-R-V}, either $L'=L$ or $A\,L'=L$. In this case $G$ is an ascending HNN-extension of $L$, which easily yields that $G$ embeds in the direct product $AGL(n,\Q) \times \Z$, where the homomorphism $G \to \Z$, onto the second factor, is given by the natural projection  sending $L$ to $0$ and $t$ to $1$. This again shows that $G$ is linear over $\Q$.
\end{proof}

\section{Free products with amalgamation}\label{sec:amalgams}

Just as a cyclic group embeds as an index two subgroup of a dihedral
group, many of the groups $G(A,L')$ can be embedded as index two
subgroups of groups expressed as free products with amalgamation.

\begin{thm} \label{thm:fpa}
Let $L = \Z^n$, let $A\in GL(n,\Q)$  and
let $L'$ be a finite-index subgroup of $L\cap A^{-1}\,L$. Suppose that
there is a matrix $R\in GL(n,\Z)$ with the following properties.

\begin{itemize}
\item[(i)] $R^2=I_n$;

\item[(ii)] $R A R= A^{-1}$;

\item[(iii)] $R\,L'=L''$, where $L'' \coloneq A\,L' \leqslant L$.

\end{itemize}

Then the group $G(A,L')$, defined by \eqref{eq:pres_of_G(A,L')},
embeds as an index two subgroup of an amalgamated free product
$K=H*_{L'=L'}H'$, where $H$ is an index $2$ overgroup of $L$ and
$H'$ is an index $2$ overgroup of~$L'$.

If $n=2$ and $A$
is conjugate to an
element of $SO(2)$ in $GL(2,\R)$ then $K$ is \cat{}.
\end{thm}

\begin{proof}
Define the matrix $R'\coloneq RA \in GL(n,\Q)$.  From conditions (i)--(iii) it is immediate that $R'^2=I_n$, and that $R'\,L'=L'$.
The group $H$ is defined as an extension with kernel $L$ and the
quotient cyclic of order two, generated by $\rho$, say, where
conjugation by $\rho$ acts as multiplication by the matrix $R$.
Similarly, $H'$ is defined as an extension with kernel $L'$
and the quotient cyclic of order two, generated by an element
$\rho'$ that acts on $L'$ as multiplication by $R'$. Let $K\coloneq H*_{L'=L'}H'$ be the amalgamated product of $H$ and $H'$ along their common subgroup $L'$.

Let $r \in H$ and $r' \in H'$ be some preimages of $\rho$ and $\rho'$ respectively.
Let us first check that the subgroup $M\coloneq \langle L, rr'\rangle$  has index $2$ in $K$. Evidently $K$ is generated by $M$ and $r$, and
$rLr^{-1}=L$, $r(rr')r^{-1}=r^2 r'^2 r'^{-1}r^{-1} \in L (rr')^{-1} \subseteq M$, as $r^2 \in L$ and $r'^2 \in L' \subseteq L$ in $K$.
Therefore $r M r^{-1} \subseteq M$, which implies that
$M \lhd K$ because $r^2 \in M$. It follows that $M$ is the kernel of the epimorphism
$\eta: K \to \Z/2\Z$, defined by $\eta(r)=\overline{1}$, $\eta(r')=\overline{1}$  and $\eta(L)=\{\overline{0}\}$. Thus $|K:M|=2$, as claimed.

Let $T$ be the Bass-Serre tree for the decomposition of $K$ as an amalgamated free product.
The vertices of $T$ can be identified with the cosets
$cH$ or $cH'$, $c \in K$, and the edges correspond to the cosets $cL'$, $c \in K$. Let $u$, $v$ be the vertices of $T$ corresponding to $H$ and $H'$ respectively; let $e$ be the edge of $T$ corresponding to $L'$.
Then $e$ joins $u$ with $v$ in $T$, and this edge, together with the endpoints, is a fundamental domain for the action of $K$ on $T$.
As a  fundamental domain for the induced action of $M$ on $T$ we can take the union $e \cup e'$, where $e'\coloneq r'^{-1} \,e$
with the vertices $u=e_-$, $v=e_+=e'_+$ and $w\coloneq e'_-=r'^{-1}\,e_-$.
Note that $w=r'^{-1}\,u=r'^{-1}r^{-1}\,u \in M\,u$, since $r\,u=u$ and $rr' \in M$.
Thus the action of $M$ on $T$ has two orbits of vertices and two orbits of edges, and
the quotient graph $M \backslash T$ consists of two vertices and two edges joining them.

Observe that the $M$-stabilizers of $u$, $v$, $e$ and $e'$ are $L$, $L'$, $L'$ and $r'^{-1}L'r'=L'$ respectively. We can now apply the Structure Theorem of Bass-Serre theory \cite[Section I.5.4]{Serre}
to find that $M$ has the following presentation:
\begin{equation}\label{eq:pres_of_M}
M=\langle L,t \,\|\, tct^{-1}=(r'^{-1}r^{-1})^{-1}c(r'^{-1}r^{-1}),~\forall\,c \in L'\rangle.
\end{equation}
It remains to observe that for each $c \in L'$ we have $(r'^{-1}r^{-1})^{-1}c(r'^{-1}r^{-1})=(rr')c(rr')^{-1}=A\,c$, hence the presentation
\eqref{eq:pres_of_M}, of $M$, coincides with the presentation \eqref{eq:pres_of_G(A,L')}, of $G(A,L')$. Thus $M \cong G(A,L')$.

Now suppose that $n=2$ and $A\in SO(2)$.  The cases $A=\pm I_2$ are easily dealt with.
Assuming that $A\in SO(2)$ has order at least $3$, a direct computation shows that any matrix $R \in SL^{\pm}(2,\R)$ satisfying $RA=A^{-1} R$ is a reflection matrix from $O(2)$.
Hence whenever $A$ preserves some inner product on $L\otimes\R$,
that inner product is also preserved by $R$ and $R'$.  Once one knows
this, showing that $K$ is a \cat{} group is similar to the
proof of Theorem~\ref{thm:iscat0}.  There are isometric actions of
the group on $\E^2$ and on $T$, the Bass-Serre tree for the given
decomposition as a free product with amalgamation.  Furthermore,
the diagonal action on $\E^2\times T$ is
properly discontinuous, cocompact and by isometries.  (If we metrize
$T$ so that each edge has length $1/2$, the space $\E^2\times T$ with
its action of $G(A,L')$ is equivariantly isometric to the space
constructed in Theorem~\ref{thm:iscat0}.)
\end{proof}

\begin{rem}
In the case when $L'= L\cap A^{-1}\,L$, one has that
$R\,L'= R\,L\cap R A^{-1}\,L = L\cap A R\,L= L\cap A\,L=A\,L'$,
and so in this case condition~(iii) follows from conditions
(i)~and~(ii).
\end{rem}

The possible isomorphism types of the group $H$ arising in
Theorem~\ref{thm:fpa} depend on the action of $R$ on $L$.
To make this precise, we need a further definition.

Let $L$ be a free abelian group of rank two, and let $\rho$ be an
involution of $L$ that reverses the orientation of $L$.  It
can be shown that there are two conjugacy classes of such involutions
in $GL(2,\Z)$, depending on whether $L$
has a basis which is permuted by $\rho$ or has a basis of eigenvectors
for $\rho$.  We refer to the first as the {\sl rhombic} case and to the
second as the {\sl rectangular} case.  If one
chooses an inner product on $L \otimes\R$ that is preserved
by the action of the involution, the rhombic case corresponds to the
existence of a basis for $L$ consisting of vectors of the same
length swapped by $\rho$, and the rectangular case corresponds to the
existence of an orthogonal basis for $L$ consisting of eigenvectors
for $\rho$.

\begin{prop} Let $H$ be a group expressed as an extension with kernel
$L\cong \Z^2$ and quotient cyclic of order two, generated by $\rho$.
If the action of $\rho$ on $L$ is rhombic then $H$ is isomorphic to
the wreath product $\Z\wr C_2$.  If the action of $\rho$ on $L$
is rectangular then $H$ is isomorphic either to the direct product
$\Z\times D_\infty$ or to the Klein bottle group $BS(1,-1)$.
\end{prop}

\begin{proof}
Group extensions with quotient cyclic of order two and kernel a
given $C_2$-module $V$ are classified by $H^2(C_2;V)$.  In the
rhombic case $L$ is a free $\Z \langle \rho\rangle$-module and so
there is only the split extension $\Z\wr C_2$.  In the rectangular
case $H^2(\langle \rho\rangle;L)$ has order two.  The split
extension gives $\Z\times D_\infty$ and the non-split extension
gives the Klein bottle group $BS(1,-1)$.
\end{proof}

\begin{ex}\label{ex:embed}
Next we consider how this embedding result applies to the groups
$G_{k,m}$, $G_P$ and $G'_P$ which were defined in Example~\ref{ex:groups}.

Let $\{a,b\}$ be the standard basis for $L=\Z^2$.  For each $k$ and $m$,
the matrix $A_{k/m}$, defined in \eqref{eq:A_k/m}, is inverted by the matrix
$R= \begin{pmatrix} 0& 1 \\ 1 & 0 \\ \end{pmatrix}$.  If $L'$ is the submodule
of $L$ spanned by $a$ and $mb$, then $RA\,a=a$ and
$RA\,mb=-mb+ka$.  A calculation shows that the action of $R$
on $L=\Z^2$ is rhombic and the action of $RA$ on $L'$ is rhombic in
the case when $k$ is odd and is rectangular in the case when $k$ is
even.  Thus in each case the group $G_{k,m}$ embeds as an index two
subgroup of a group $H*_{L'=L'}H'$ as in the statement of
Theorem~\ref{thm:fpa}, where  $H\cong \Z\wr C_2$.  If $k$ is odd, then $H'$
is also isomorphic to $\Z\wr C_2$, whereas if $k$ is even, then
$H'$ may be taken to be either $\Z\times D_\infty$ or $BS(1,-1)$.

Since the matrix $A_P$, defined in \eqref{eq:A_P}, is already in
$SO(2)$, there are four possible choices for $R \in O(2) \cap
GL(2,\Z)$: $\begin{pmatrix} 0& 1 \\ 1 & 0  \end{pmatrix}$ or
$\begin{pmatrix} 0& -1 \\ -1 & 0  \end{pmatrix}$ (rhombic case), or
$\begin{pmatrix} 1& 0 \\ 0 & -1 \end{pmatrix}$ or $\begin{pmatrix}
  -1& 0 \\ 0 & 1  \end{pmatrix}$ (rectangular case).  In the case
when $L'=\langle (2,-1)^T,(1,2)^T\rangle$ is as large as possible, a
calculation shows that the action of $R A_P$ on $L'$ and the action
of $R$ on $L$ both have the same type.  Thus we obtain five
potentially different amalgamated products of the form $H*_{L'=L'}H'$
that contain $G_P$ as an index two subgroup, including one
torsion-free group in which each of $H$ and $H'$ is isomorphic to
$BS(1,-1)$.  None of the possible choices for $R$ swaps $(5\Z)^2$ and
its image under $A_P$, so we cannot construct a group of this form
containing $G'_P$ as an index two subgroup.
\end{ex}

\begin{ex} Let us give an explicit presentation for the torsion-free amalgamated product
$K=H*_{L'=L'}H'$, where  $A=A_{P}=\begin{pmatrix} 3/5& -4/5 \\ 4/5 & 3/5  \end{pmatrix}$, $L'=\langle (2,-1)^T,(1,2)^T\rangle$
and $R=\begin{pmatrix} -1& 0 \\ 0 & 1 \\ \end{pmatrix}$ (rectangular case). Then the group $G_P$ from Example~\ref{ex:groups} embeds in $K$ as a subgroup of index $2$, and
$K \cong BS(1,-1)*_{\Z^2} BS(1,-1)$ has the presentation
\[ \langle a,r,c,s\,\|\, rar^{-1}=a^{-1},\,\, scs^{-1}=c^{-1},\,\, a^2r^{-2}=c,\,\, ar^4=s^2 \rangle,
\]
which can be transformed to the $2$-generator and $2$-relator presentation
\[ \langle r,s\,\|\, r^{5}s^{-2}r^{3}s^{-2}=1,\,\, s^5r^{-10}s^3r^{-10}=1 \rangle.
\]
\end{ex}

\begin{rem}\label{rem:am_prod_non-Hopf}
Suppose that the hypotheses of both Proposition~\ref{prop:nonhopf}
and Theorem~\ref{thm:fpa} apply, and we are in the split case (i.e., when $r^2=r'^2=1$, $H=L \rtimes \langle r\rangle$
and $H'=L'\rtimes \langle r'\rangle$). Then the endomorphism
$\phi$ constructed
in the proof of Proposition~\ref{prop:nonhopf} extends to the amalgamated product
$H*_{L'=L'}H'$ via $\phi(r)=r$ and $\phi(r')=r'$.  This gives examples of free
products of virtually $\Z^2$ groups amalgamating abelian subgroups
that are non-Hopfian, and, in particular, not residually
finite.
\end{rem}

\begin{cor}\label{cor:am_prod_non-biaut}
In addition to the assumptions of Theorem \ref{thm:fpa}, suppose that
$A$ has infinite order. Then the amalgamated product $K=H*_{L'=L'}H'$
is not virtually biautomatic.
\end{cor}

\begin{proof} Indeed, the group $G=G(A,L')$ is not virtually biautomatic by
Theorem~\ref{thm:notvirtbiaut} and, since $G$ has index $2$ in $K$,
$K$ cannot be virtually biautomatic by \cite[Theorem 4.1.4]{Eps}.
\end{proof}

Baumslag, Gersten, Shapiro and Short \cite{BGSS} proved that a free
product of two finitely generated free abelian groups amalgamating any
subgroup is always automatic. However, Bridson \cite[Remark~3]{Bridson-growth} noted that amalgamated products of finitely generated virtually abelian groups can be much wilder; in particular they may not even be asynchronously automatic.
Corollary~\ref{cor:am_prod_non-biaut} together with Example~\ref{ex:embed} produce examples of similar spirit.

\section{Closing remarks and open questions}\label{sec:open-q}
Motivated by the question whether all automatic groups are biautomatic, we tried checking automaticity of the groups $G_P$ and $G_{k,m}$ from Example~\ref{ex:groups}. As part of this effort, we used the GAP implementation of Holt's KBMAG software~\cite{kbmag,gap} to search for automatic structures on the groups $G_{k,2}$ as defined
in Example~\ref{ex:groups} for various choices of $k$.  In the
cases $k\in\{-2,0,2\}$ it easily found an automatic structure, while
for $k\in\{-3,-1,1,3,4,5\}$ it failed to find any automatic structure.

Throughout the rest of this section we assume that $n \in \N$, $L=\Z^n$, $A \in GL(n,\Q)$ and $L'$ is a finite-index subgroup of $L$ such that $A\,L' \subseteq L$. The questions below concern the groups $G(A,L')$, defined in \eqref{eq:pres_of_G(A,L')}.

\begin{question} Can the group $G(A,L')$ be automatic when $A$ has infinite order?
\end{question}

As explained in the introduction, the groups $G(A,L')$ are higher-dimensional analogues of Baumslag-Solitar groups $BS(k,l)$, which cannot be subgroups of biautomatic groups unless $|k|=|l|$ \cite[Proposition~6.7]{G-S}. This yields the following question, suggested to the authors by K.-U. Bux:

\begin{question} Suppose that $G(A,L')$ embeds as a subgroup in a biautomatic group. Does it follow that $A$ has finite order?\footnote{A positive answer to this question has recently been announced by M. Valiunas, \texttt{{arXiv:2104.13688}}.}
\end{question}

In \cite[Question~43]{FHT} Farb, Hruska and Thomas  asked whether every group which acts properly and cocompactly on a \cat{} piecewise Euclidean $2$-complex is
biautomatic. This question remains open, though one can show that our groups $G_P$ and $G_{k,m}$, $-2m<k<2m$, admit geometric actions on  \cat{} piecewise Euclidean $3$-complexes.

Andreadakis, Raptis and Varsos \cite{A-R-V-2} described a necessary and sufficient criterion for an HNN-extension of $\Z^n$ to be Hopfian. However, it is not obvious how to check this criterion for any given group $G(A,L')$. It would be interesting to characterize the Hopficity of $G(A,L')$ only in terms of the matrix $A$ and the finite-index subgroup $L'$ of $L=\Z^n$
(in the spirit of the residual finiteness criterion from Proposition~\ref{prop:rf-crit}).

\begin{question}
Classify the groups $G(A,L')$ up to isomorphism\footnote{This has recently been completed by M. Valiunas in \texttt{arXiv:2011.08143}.}, commensurability and quasi-isometry.
\end{question}

We expect that the classification up to isomorphism should be straightforward, while the classification up to commensurability will be more challenging (it has only recently been
completed by Casals-Ruiz, Kazachkov and Zakharov \cite{CR-K-Z} for the Baumslag-Solitar groups, which correspond to the case $n=1$). For $n=1$ the quasi-isometry classification of Baumslag-Solitar groups was done by Whyte~\cite{Whyte-BS}. Suppose that $A$ is conjugate to an orthogonal matrix in $GL(n,\R)$. Then, in view of  Corollary~\ref{cor:quasi-isom}, $G(A,L')$ is quasi-isometric to the direct product of $\Z^n$ with a free group. So, for a fixed $n \in \N$, there are exactly $2$ quasi-isometry classes for such $A$: in the first class $L'=L$, $A$ has finite order and $G(A,L')$ is virtually $\Z^{n+1}$, and in the second class $L'$ is a proper subgroup of $L$ and $G(A,L')$ is quasi-isometric to $\Z^n \times F_2$, where $F_2$ is the free group of rank $2$. In the latter case $G(A,L')$ is commensurable with $\Z^n \times F_2$ if and only if $A$ has finite order.

Finally, an observant reader may notice that, according to Theorem~\ref{thm:ashtalk}, if the group $G(A,L')$ is both \cat{} and residually finite then it is biautomatic. Thus our methods leave the following question open.

\begin{question}
Is every residually finite \cat{} group biautomatic?
\end{question}


\begin{thebibliography}{99}

\bibitem{alobri} J.M. Alonso, M.R. Bridson, Semihyperbolic groups.
\emph{Proc. London Math. Soc. (3)} \textbf{70} (1995), 56--114.

\bibitem{Amrhein} A. Amrhein, Characterizing Biautomatic Groups. Preprint (2021). \texttt{arXiv:2105.07509}

\bibitem{A-R-V-2}  S. Andreadakis, E. Raptis, D. Varsos, Hopficity of certain HNN-extensions.  \emph{Comm. Algebra} \textbf{20} (1992), no. 5, 1511--1533.

\bibitem{A-R-V}  S. Andreadakis, E. Raptis, D. Varsos, Residual finiteness and Hopficity of certain HNN extensions.
\emph{Arch. Math. (Basel)} \textbf{47} (1986), no. 1, 1--5.

\bibitem{ABLRSV}  G.N. Arzhantseva, J. Burillo, M. Lustig, L. Reeves, H. Short, E. Ventura,  Uniform non-amenability. \emph{Adv. Math.} \textbf{197} (2005), no. 2, 499--522.

\bibitem{BGSS}  G. Baumslag, S.M. Gersten, M. Shapiro, H. Short, Automatic groups and amalgams. \emph{J. Pure Appl. Algebra} \textbf{76} (1991), no. 3, 229--316.


\bibitem{Baum-Sol} G. Baumslag, D. Solitar,
Some two-generator one-relator non-Hopfian groups.
\emph{Bull. Amer. Math. Soc.} \textbf{68} (1962), 199--201.

\bibitem{Bestvina-list} M. Bestvina, Questions in Geometric Group Theory. \url{http://www.math.utah.edu/~bestvina/eprints/questions-updated.pdf}

\bibitem{Bridson-combings} M.R. Bridson, Regular combings, nonpositive
curvature and the quasiconvexity of abelian subgroups.  \emph{J. Pure
Appl. Algebra} \textbf{88} (1993), no. 1-3, 23--35.

\bibitem{Bridson-growth}  M.R. Bridson, On the growth of groups and automorphisms. \emph{Internat. J. Algebra Comput.} \textbf{15} (2005), no. 5-6, 869--874.

\bibitem{bri-gil} M.R. Bridson, R. Gilman, Formal language theory
and the geometry of $3$-manifolds. \emph{Comment. Math. Helv.} \textbf{71}
(1996), no. 4, 525--555.

\bibitem{brihae} M.R. Bridson, A. Haefliger, Metric spaces of
non-positive curvature.  \emph{Springer-Verlag, Berlin}, 1999.  xxi+643
pp.



\bibitem{capmon} P.-E. Caprace, N. Monod,
Isometry groups of non-positively curved spaces: discrete subgroups.
\emph{J. Topol.} \textbf{2} (2009), no. 4, 701--746.

\bibitem{capmon-corr} P.-E. Caprace, N. Monod, Erratum and addenda to ``Isometry groups of non-positively curved spaces: discrete subgroups''. Preprint (2019). \texttt{arXiv:1908.10216}

\bibitem{CR-K-Z} M. Casals-Ruiz, I. Kazachkov, A. Zakharov, Commensurability of Baumslag-Solitar groups. Preprint (2019). \texttt{arXiv:1910.02117}



%\bibitem{C-M} M. Culler, J.W. Morgan, Group Actions on $\mathbb{R}$-Trees. \emph{Proc. London Math. Soc. (3)} \textbf{55}
%(1987), no. 3, 57--604.

\bibitem{Elder} M. Elder, Automaticity, almost convexity and falsification by fellow traveler properties of some finitely generated groups. PhD thesis, \emph{The University of Melbourne (2000).}

\bibitem{Eps} D.B.A. Epstein, J.W. Cannon, D.F. Holt, S.V.F. Levy,
  M.S. Paterson, W.P. Thurston,
Word processing in groups. \emph{Jones and Bartlett Publishers,
  Boston, MA,} 1992. xii+330 pp.


\bibitem{kbmag}  D.B.A. Epstein, D.F. Holt, S.E. Rees, The use of
Knuth-Bendix methods to solve the word problem in automatic
groups.  \emph{J. Symbolic Computation} \textbf{12} (1991), 397--414.

\bibitem{FHT} B. Farb, C. Hruska, A. Thomas, Problems on
automorphism groups of nonpositively curved polyhedral complexes
and their lattices.  A chapter in `Geometry, rigidity and group
actions' Chicago Lectures in Math., \emph{Univ. Chicago Press,
Chicago}, 2011, 515--560.




\bibitem{G-S} S.M. Gersten, H.B. Short, Rational subgroups of
  biautomatic groups.
\emph{Ann. Math. (2)} \textbf{134} (1991), no. 1, 125--158.

\bibitem{grhe} J.R.J. Groves, S.M. Hermiller,
Isoperimetric inequalities for soluble groups.
\emph{Geom. Dedicata} \textbf{88} (2001), no. 1--3, 239--254.


\bibitem{huangprytula} J. Huang, T. Prytu\l{a}, Commensurators
of abelian subgroups in \cat{} groups. \emph{Math. Z.} \textbf{296} (2020), no. 1-2, 79--98.

\bibitem{krop} P.H. Kropholler, Baumslag-Solitar groups and some
other groups of cohomological dimension two.  \emph{Comment. Math.
Helv.} \textbf{65} (1990), 547--558.

\bibitem{L-S}  R.C. Lyndon, P.E. Schupp, Combinatorial group theory.
Ergebnisse der Mathematik und ihrer Grenzgebiete, Band 89. \emph{Springer-Verlag, Berlin-New York}, 1977. xiv+339 pp.

\bibitem{macbir} S. Mac~Lane, G. Birkhoff, Algebra, Second
edition.  \emph{Macmillan Inc, New York}, 1979. xv+586 pp.

\bibitem{Mal} A. Malcev,
On isomorphic matrix representations of infinite groups. (Russian)
\emph{Mat. Sbornik} \textbf{8 (50)} (1940), 405--422.

\bibitem{mccammond} J. McCammond, Algorithms, Dehn functions and
automatic groups. Preprint (2007). \url{https://sites.google.com/a/scu.edu/rscott/pggt}

\bibitem{meier} D. Meier, Non-Hopfian groups.
\emph{J. London Math. Soc. (2)} \textbf{26} (1982), no. 2, 265--270.

\bibitem{Min-fibre} A. Minasyan, On conjugacy separability of fibre products. \emph{Proc. Lond. Math. Soc. (3)} \textbf{115} (2017), no 66,  1170--1206.

\bibitem{Min-VR} A. Minasyan, Virtual retraction properties in groups. \emph{IMRN}, to appear.  \\ \texttt{arXiv:1810.02654}

\bibitem{Neu-Sha}  W.D. Neumann, M. Shapiro, Equivalent automatic structures and their boundaries. \emph{Internat. J. Algebra Comput.} \textbf{2} (1992), no. 4, 443--469.


\bibitem{nibree} G.A. Niblo, L. Reeves, The geometry of
cube complexes and the complexity of their fundamental groups.
\emph{Topology} \textbf{37} (1998), no. 3, 621--633.

\bibitem{osajdaprytula} D. Osajda, T. Prytu\l{a}, Classifying
spaces for families of subgroups for systolic groups.
\emph{Groups Geom. Dyn.} \textbf{12} (2018), 1005--1060.

\bibitem{Osin-Kazhdan} D.V. Osin,
Kazhdan constants of hyperbolic groups. (Russian)
\emph{Funktsional. Anal. i Prilozhen.} \textbf{36} (2002), no. 4, 46--54. English translation in \emph{Funct. Anal. Appl.} \textbf{36} (2002), no. 4, 290--297.

\bibitem{prytula} T. Prytu\l{a}, Bredon cohomological dimension for
virtually abelian stabilizers for \cat{} groups.   \emph{J.  Topol. Anal.}, to appear. \texttt{arXiv:1810.08924}

\bibitem{Serre} J.-P. Serre, Trees. Translated from the French original by J. Stillwell. Corrected 2nd printing of the 1980
English translation. Springer Monographs in Mathematics, \emph{Springer-Verlag, Berlin,} 2003. x+142 pp.

\bibitem{gap} The GAP Group,
\emph{GAP--Groups, Algorithms and Programming},
Version~4.9.2 (2018). \url{https://www.gap-system.org/}

\bibitem{Whyte-BS} K. Whyte, The large scale geometry of the higher
Baumslag-Solitar groups.  \emph{Geom. Funct. Anal.} \textbf{11}
(2001), no. 6, 1327--1343.

\bibitem{Wise-non-Hopf} D.T. Wise,
A non-Hopfian automatic group.
\emph{J. Algebra} \textbf{180} (1996), no. 3, 845--847.

\bibitem{Wise-thesis}  D.T. Wise, Non-positively curved squared complexes, aperiodic tilings and non-residually finite groups. Ph.D. thesis, \emph{Princeton University (1996).}


\end{thebibliography}
\end{document}